\documentclass[12pt,a4paper]{amsart}
\usepackage[margin=2.5cm]{geometry}
\usepackage{hyperref}

\usepackage{stmaryrd}
\usepackage{multirow}
\usepackage[normalem]{ulem}

\usepackage{mathrsfs}

\usepackage[abbrev]{amsrefs}

\usepackage{bbold}
\usepackage{amssymb}
\usepackage{amsmath}
\usepackage{amsthm}
\usepackage{braket}
\usepackage{paralist}
\usepackage{eufrak}
\usepackage[all,cmtip]{xy}
\usepackage{rotating}
\usepackage{arydshln}

\usepackage{xcolor}
\usepackage{mathrsfs}
\usepackage{mathbbol}

\newcommand{\tG}{\widetilde{G}}

\newcommand{\cqq}{\mathscr{D}}
\newcommand{\Sym}{\mathrm{Sym}}
\newcommand{\rSp}{\mathrm{Sp}}
\newcommand{\rO}{\mathrm{O}}

\newcommand{\frakp}{\mathfrak{p}}
\newcommand{\pr}{\mathrm{pr}}

\newcommand{\Rad}{\mathrm{Rad}}

\newcommand{\Hol}{\mathrm{Hol}}
\newcommand{\AC}{\mathrm{AC}}
\newcommand{\AV}{\mathrm{AV}}
\newcommand{\VC}{\mathrm{V}_\bC}

\newcommand{\depth}{\mathrm{depth}}
\newcommand{\wtM}{\widetilde{M}}
\newcommand{\wtMone}{{\widetilde{M}^{(1,1)}}}

\newcommand{\nullpp}{N(\fpp'^*)}
\newcommand{\nullp}{N(\fpp^*)}
\newcommand{\Aut}{\mathrm{Aut}}

\newcommand{\bN}{\mathbb{N}}

\newcommand{\bfB}{\mathbf{B}}
\newcommand{\bfA}{\mathbf{A}}
\newcommand{\bfE}{\mathbf{E}}

\newcommand{\bfdd}{\mathbf{d}}

\newcommand{\tg}{{\tilde{g}}}
\newcommand{\tk}{{\tilde{k}}}

\newcommand{\sfg}{{\mathsf{g}}}
\newcommand{\sfk}{{\mathsf{k}}}
\newcommand{\sfp}{{\mathsf{p}}}
\newcommand{\sfG}{{\mathsf{G}}}
\newcommand{\sfK}{{\mathsf{K}}}

\def\Wcp{W}

\providecommand{\bcN}{{\overline{\cN}}}
\providecommand{\pcN}{{\partial\cN}}
\providecommand{\bcO}{{\overline{\cO}}}

\def\bcOp{\overline{\cO'}}

\def\frakN{\mathfrak{N}}
\def\topform{\mbox{$\bigwedge^{\! \mathrm{top} \, }$}}

\def\barxi{{\overline{\xi}}}
\def\Stab{{\rm Stab}}

\def\Ad{{\rm Ad}}

\def\inn#1#2{\left\langle{#1},{#2}\right\rangle}

\def\Sp{{\rm Sp}}
\def\sp{{\mathfrak{sp}}}
\def\rSp{{\rm Sp}}

\def\O{{\rm O}}
\def\rO{{\rm O}}

\def\det{{\rm det}}

\def\wtE{\widetilde{E}}
\def\wtG{\widetilde{G}}

\def\wtK{\widetilde{K}}

\def\wtM{\widetilde{M}}

\def\fgg{\mathfrak{g}}

\def\fkk{\mathfrak{k}}
\def\fmm{\mathfrak{m}}
\def\fpp{\mathfrak{p}}
\def\fss{\mathfrak{s}}

\def\sA{\mathscr{A}}
\def\sB{\mathscr{B}}

\def\sH{\mathscr{H}}
\def\sL{\mathscr{L}}

\def\sO{\mathscr{O}}
\def\sQ{\mathscr{Q}}

\def\sY{\mathscr{Y}}
\def\sZ{\mathscr{Z}}

\def\rT{{\rm{T}}}
\def\rP{{\rm{P}}}
\def\rU{{\rm{U}}}
\def\rV{{\mathrm{V}}}

\def\cH{\mathcal{H}}
\def\cN{\mathcal{N}}
\def\cC{\mathcal{C}}
\def\cD{\mathcal{D}}
\def\cO{\mathcal{O}}
\def\cS{\mathcal{S}}
\def\cU{\mathcal{U}}

\def\bC{{\mathbb{C}}}
\def\bH{{\mathbb{H}}}
\def\bR{{\mathbb{R}}}
\def\GL{\mathrm{GL}}
\def\wtSp{\widetilde{\mathrm{Sp}}}
\def\wtrU{\widetilde{\mathrm{U}}}
\def\fsp{{\mathfrak{sp}}}

\def\Ann{{\mathrm{Ann}\,}}
\def\Ker{{\rm Ker}\,}
\def\Lie{{\rm Lie}}

\def\Hom{{\rm Hom}}

\def\Ind{{\rm Ind}}

\def\Spec{{\rm Spec\,}}
\def\Supp{\mathrm{Supp}}
\def\Sym{\mathrm{Sym}}
\def\Alt{\mathrm{Alt}}
\def\Gr{\mathrm{Gr\,}}
\def\codim{\mathrm{codim}}

\def\wtMoo{{{\wtM^{(1,1)}}}}

\newtheorem{thmA}{Theorem}

\newtheorem{propA}[thmA]{Proposition}
\newtheorem{corA}[thmA]{Corollary}

\newtheorem{thm}{Theorem}[section]
\newtheorem{lemma}[thm]{Lemma}
\newtheorem{prop}[thm]{Proposition}

\newtheorem*{prop*}{Proposition}
\theoremstyle{definition}
\newtheorem*{definition}{Definition}

\author{Hung Yean Loke}

\address{Department of Mathematics,
National University of Singapore,
Block S17,
10, Lower Kent Ridge Road,
Singapore 119076}
\email{matlhy@nus.edu.sg}

\author{Jiajun Ma}

\address{Department of Mathematics, 
Ben-Gurion University of the Negev,
P.O.B. 653, Be'er Sheva 84105,
Israel}

\email{jiajun@math.bgu.ac.il}

\subjclass[2010]{22E46, 22E47}

\keywords{local theta lifts, nilpotent orbits, associated varieties,
  associated cycles, moment maps}

\allowdisplaybreaks

\begin{document}
\title{Invariants and $K$-spectrums of local theta lifts}

\begin{abstract}
  Let $(G,G')$ be a type I irreducible reductive dual pair in
  $\Sp(W_\bR)$. We assume that $(G,G')$ is in the stable range where
  $G$ is the smaller member. Let $K$ and $K'$ be maximal compact
  subgroups of $G$ and $G'$ respectively. Let $\fgg = \fkk \oplus
  \fpp$ and $\fgg' = \fkk' \oplus \fpp'$ be the complexified Cartan
  decompositions of the Lie algebras of $G$ and $G'$ respectively.
  Let $\wtK$ and $\wtK'$ be the inverse images of $K$ and $K'$ in the
  metaplectic double cover $\wtSp(W_\bR)$ of $\Sp(W_\bR)$.  Let $\rho$
  be a genuine irreducible \mbox{$(\fgg,\wtK)$-module}. Our first main
  result is that if~$\rho$ is unitarizable, then except for one
  special case, the full local theta lift $\rho' = \Theta(\rho)$ is
  equal to the local theta lift $\theta(\rho)$. Thus excluding the
  special case, the full theta lift~$\rho'$ is an irreducible and
  unitarizable $(\fgg',\wtK')$-module. Our second main result is that
  the associated variety and the associated cycle of~$\rho'$ are the
  theta lifts of the associated variety and the associated cycle of
  the contragredient representation $\rho^*$ respectively. Finally we
  obtain some interesting $(\fgg,\wtK)$-modules whose $\wtK$-spectrums
  are isomorphic to the spaces of global sections of some vector
  bundles on some nilpotent $K_\bC$-orbits in~$\fpp^*$.
\end{abstract}

\maketitle

\section{Introduction}

\subsection{} \label{S11} Let $W_\bR$ be a finite dimensional
symplectic real vector space. Throughout this paper $(G,G')$ will
denote a type I irreducible reductive dual pair in $\Sp(W_\bR)$.  Such
dual pairs are listed in Table~\ref{tab:sstabncp} in
Section~\ref{S21}.

We follow the notation in \cite{Howe89} closely. Let $\wtSp(W_\bR)$ be
the metaplectic double cover of $\Sp(W_\bR)$. For any subgroup $E$ of
$\Sp(W_\bR)$, let $\wtE$ denote its inverse image in $\wtSp(W_\bR)$.
We choose a maximal compact subgroup $\rU$ of $\Sp(W_\bR)$ such that
$K:=G \cap \rU$ and $K':=G' \cap \rU$ are maximal compact subgroups of
$G$ and $G'$ respectively. Hence $\wtK$ and $\wtK'$ are maximal
compact subgroups of $\wtG$ and $\wtG'$ respectively.  The choice of U
determines a unique complex structure on $W_\bR$ with the following
property: there is a positive definite Hermitian form $\langle \, , \,
\rangle$ on the resulting complex vector space $W$ so that the
imaginary part of $\langle \, , \, \rangle$ coincides with the
symplectic form on $W_\bR$, and $\rU$ coincides with the unitary group
attached to $(W, \langle \, , \, \rangle)$.  We choose the oscillator
representation of $\wtSp(W_\bR)$ whose Fock model $\sY$ is realized as
the space $\bC[W]$ of complex polynomials on $W$ with the $\wtrU$
action as described in Appendix~\ref{sec:Fock}.  Let $\varsigma$
denote the minimal $\widetilde{\rU}$-type of $\sY$. It is a one
dimensional representation of $\wtrU$ acting on the space of constant
functions in $\bC[W]$. Let $\varsigma|_{\wtE}$ denote the restriction
of $\varsigma$ to $\wtE$ for any subgroup $E$ of $\rU$.

Let $\fgg = \fkk \oplus \fpp$ denote the complexified Cartan
decomposition of the Lie algebra of $\wtG$ corresponding to the
maximal compact subgroup $\wtK$.  Likewise we define $\fgg' = \fkk'
\oplus \fpp'$ for~$\wtG'$. Let $\rho$ be an irreducible admissible
genuine $(\fgg,\wtK)$-module.  By (2.5) in \cite{Howe89},
\[
\sY / (\cap_{\psi \in \Hom_{\fgg,\wtK}(\sY, \rho)} \ker \psi)
\simeq \rho \otimes \Theta(\rho)
\]
where $\Theta(\rho)$ is a $(\fgg',\wtK')$-module called the {\it full
  (local) theta lift} of~$\rho$. Theorem~2.1 in~\cite{Howe89} states
that if $\Theta(\rho) \neq 0$, then $\Theta(\rho)$ is a
$(\fgg',\wtK')$-module of finite length with an infinitesimal
character and it has a unique irreducible quotient~$\theta(\rho)$
called the {\it (local) theta lift} of $\rho$. Moreover if
$\theta(\rho_1)$ and $\theta(\rho_2)$ are nonzero, then they are
isomorphic if and only if $\rho_1$ and $\rho_2$ are isomorphic.

\medskip

It is a result of Protsak and Przebinda~\cite{PPz} that in the stable
range, $\theta(\rho)$ is nonzero.  This partially generalizes a
previous result of Li \cite{Li1989} which states that if $\rho$ is
irreducible and unitarizable, then $\theta(\rho)$ is nonzero and
unitarizable.

In order to state our first result, we exclude following
special case.
\begin{itemize}
\item[$(\dagger)$] The dual pair $(G,G') = (\Sp(n,\bR),\O(2n,2n))$
  and $\rho$ is the one dimensional genuine representation of
  $\widetilde{\Sp}(n,\bR)$.
\end{itemize}

\begin{thmA} \label{TA} Suppose $(G,G')$ is in the stable
  range where $G$ is the smaller member. Let $\rho$ be an irreducible
  unitarizable genuine $(\fgg,\wtK)$-module. We exclude the
  case~$(\dagger)$ above.  Then
\[
\Theta(\rho) = \theta(\rho)
\]
as $(\fgg',\wtK')$-modules.  In other words, $\Theta(\rho)$ is already
irreducible and unitarizable.
\end{thmA}

The proof is given in Section \ref{S23a}. 

In Case $(\dagger)$, $\Theta(\rho)$ is reducible by 
  Lee's appendix in \cite{Lo}.

  The above theorem is useful because invariants attached to
  $\Theta(\rho)$ are usually easier to describe than that of
  $\theta(\rho)$. For example, we could deduce a formula for the
  $\wtK'$-types of~$\Theta(\rho)$ in Proposition \ref{P4}.

\subsection{} \label{S12} 
Before stating other results, we briefly review the definitions of
some invariants of Harish-Chandra modules. See Section~2 in Vogan
\cite{Vo89} for details.

Let $(\sfg,\sfK)$ denote the Harish-Chandra pair of a real reductive
group $\sfG$. Let $\sfg = \sfk \oplus \sfp$ denote the complexified
Cartan decomposition the Lie algebra of $\sfG$ corresponding to the
maximal compact subgroup $\sfK$. Let $(\varrho, V_{\varrho})$ be a
$(\sfg,\sfK)$-module of finite length and let $0 \subset F_0 \subset
\cdots\subset F_j \subset F_{j+1} \subset \cdots$ be a good filtration
of $\varrho$, i.e. $\dim F_j$ is finite, $\cup_{j \in \bN} F_j =
V_\varrho$ and $\cU_p(\sfg) F_q = F_{p+q}$ for all $q$ sufficiently
large and for all $p > 0$.  Then $\Gr\varrho = \bigoplus
F_{j}/F_{j-1}$ is a finitely generated $(\cS(\sfp),\sfK)$-module where
$\cS(\sfp)$ is the symmetric algebra on $\sfp$.

Let $\sA$ be the associated $\sfK_\bC$-equivariant coherent sheaf of $\Gr
\varrho$ on $\sfp^* = \Spec(\cS(\sfp))$.  The {\it associated variety}
of $\varrho$ is defined to be $\AV(\varrho): = \Supp(\sA)$ in
$\sfp^*$.  Its dimension is called the {\it Gelfand-Kirillov
  dimension} of $\varrho$.  It is a well known fact that
$\AV(\varrho)$ is a closed subset of the null cone of $\sfp^*$.

Let $\AV(\varrho) = \bigcup_{j=1}^r \overline{\cO_j}$ where $\cO_j$
are the distinct open $\sfK_\bC$-orbits in $\AV(\varrho)$.  By
Lemma~2.11 in~\cite{Vo89} (c.f. Proposition \ref{P8}), there is a
finite $(\cS(\sfp),\sfK_\bC)$-invariant filtration $0\subset \sA_0
\subset \cdots \subset \sA_l \subset \cdots \subset \sA_n = \sA$ of
$\sA$ such that $\sA_l/\sA_{l-1}$ is generically reduced on each
$\overline{\cO_j}$. For a closed point $x_j\in \cO_j$, let
$i_{x_j}\colon \{ x_j \} \hookrightarrow \sfp^*$ be the natural
inclusion map and let $\sfK_{x_j}$ be the stabilizer of $x_j$ in
$\sfK_\bC$.  Now
\[
\chi_{x_j}:= \textstyle\bigoplus_l (i_{x_j})^* (\sA_l/\sA_{l-1})
\]
is a nonzero finite dimensional rational representation of $\sfK_{x_j}$.
We call $\chi_{x_j}$ an {\it isotropy representation} of $\varrho$ at
$x_j$. Its image $[\chi_{x_j}]$ in the Grothendieck group of finite
dimensional rational $\sfK_x$-modules is called the {\it isotropy
  character} of $\varrho$ at $x_j$.  The isotropy representation is
dependent on the filtration but the isotropy character is independent
of the filtration.

We call $\Set{(\cO_j, x_j, \chi_{x_j}) : j=1, \ldots, r}$ the set of
{\it orbit data attached to the filtration} $\set{ \sA_j }$. On the
other hand, $\Set{(\cO_j,x_j, [\chi_{x_j}]) : j=1, \ldots, r}$ is
independent of the filtration and we call it the set of {\it orbit
  data attached to $\varrho$}.  Two orbit data are equivalent if they
are conjugate to each other by the $\sfK_\bC$-action.  We define the
{\it multiplicity of $\varrho$ along $\cO_j$} to be $m(\cO_j,\varrho)
= \dim_\bC\chi_{x_j}$ and the {\it associated cycle} of $\varrho$ to
be $\AC(\varrho) = \sum_{j=1}^r m(\cO_j,\varrho) [\overline{\cO_j}]$.

In summary, the associated variety, the associated cycle and isotropy
character(s) are invariants of $\varrho$, i.e. they are independent of
the choices of filtrations.  

Suppose $\sfG$ is a member group of a type I reductive dual pair in
$\wtSp(W_\bR)$. Then by~\cite{Adams}, \cite{MVW}, \cite{LST} and
\cite{Sun}, the above invariants of $\varrho$ and of its
contragredient $\varrho^*$ are related by an automorphism $C$ of
$\sfG$. We call $C$ a {\it dualizing automorphism}.  We will review
these in Appendix~\ref{sec:dual}. 

\subsection{} Now we describe a result about the associated variety
of $\Theta(\rho)$.

Let $\fgg = \fkk \oplus \fpp$ and $\fgg' = \fkk' \oplus \fpp'$ as in
Section \ref{S11}. In Appendix~\ref{SA1} (also see \cite{DKP}), we
recall the definitions of the two moment maps
\begin{equation} \label{eq2}
\xymatrix{
\fpp^* & \ar[l]_{\phi} W \ar[r]^{\phi'} & \fpp'^*.}
\end{equation}
The maps $\phi$ and $\phi'$ are given explicitly in Table
\ref{tab:noncpt1}. For a $K_\bC$-invariant closed subset~$S$
of~$\frakp^*$, we define {\it the theta lift of $S$} to be $\theta(S)
= \phi'(\phi^{-1} (S))$, which is a $K'_\bC$-invariant closed subset
of $\fpp'^*$.  Let $N(\fpp^*) := \Set{x \in \fpp^*|0\in
  \overline{K_\bC \cdot x}}$ be the nilpotent cone in $\fpp^*$.  Let
$\frakN_{K_\bC}(\fpp^*)$ be the set of nilpotent $K_\bC$-orbits
in~$\fpp^*$.  We define $\nullpp$ and $\frakN_{K'_\bC}(\fpp'^*)$ in
the same way. It is well known that $\theta(S) \subseteq N(\fpp'^*)$
if $S\subseteq N(\fpp^*)$.

\medskip

Since $\Theta(\rho)$ has finite length, the associated variety
$\AV(\Theta(\rho))$ of $\Theta(\rho)$ is a closed subvariety of $\nullpp$.

\begin{thmA} \label{TB} For any real reductive dual pair
    $(G,G')$ (not necessary in the stable range) and any
    irreducible admissible genuine $(\fgg,\wtK)$-module, there is an
    upper bound of the associated variety of $\Theta(\rho)$ given by
    $\theta(\AV(\rho^*))$. In other words, we have
\[
\AV(\Theta(\rho)) \subseteq \theta(\AV(\rho^*))).
\]
\end{thmA}

The proof is given in Section \ref{S33}.

The above theorem is a correction to Proposition 3.12 in Nishiyama-Zhu
\cite{NZ}.

\subsection{}
We now assume $(G, G')$ is in the stable range where $G$ is the
smaller member. Given $\cO\in \frakN_{K_\bC}(\fpp^*)$, it is a result
of \cite{Ohta:1991C}, \cite{Daszkiewicz:2005} and \cite{NOZ1} that
there is a unique $\cO'\in \frakN_{K'_\bC}(\fpp'^*)$ such that
$\overline{\cO'} = \theta(\bcO)$.  We call $\cO'$ the {\it theta lift
  of $\cO$} and we write $\cO' = \theta(\cO)$.  Moreover,
\[
\begin{split}
\theta \colon \frakN_{K_\bC}(\fpp^*) &\to \frakN_{K'_\bC}(\fpp'^*)\\
\cO &\mapsto \cO'
\end{split}
\]
is an injective map preserving the closure relations,
i.e. $\theta(\cO_2) \subset \overline{\theta(\cO_1)}$ if $\cO_2
\subset \overline{\cO_1}$.

\begin{definition}
  We define the following notion of theta lifts of objects in the
  stable range.
\begin{enumerate}[(1)]
\item Let $c=\sum_j m_j [\overline{\cO_j}]$ be a formal sum of
  closures of nilpotent orbits. We define the {\it theta lift of the
    cycle $c$} to be $\theta(c) := \sum_j m_j
  [\overline{\theta(\cO_j)}]$.

\item Let $(\cO, x, \chi_x)$ be an orbit datum where $\cO\in
  \frakN_{K_\bC}(\fpp^*)$, $x\in \cO$ and $\chi_x$ is a
  finite-dimensional rational $\wtK_x$-module where $\wtK_x$ is the
  stabilizer of $x$ in $\wtK_\bC$.  Let $\cO' = \theta(\cO)$.  Fixing
  points $x\in \cO$, $w\in W$, $x'\in \cO'$ such that $\phi(w) = x$
  and $\phi'(w) = x'$, we will define a group homomorphism $\alpha
  \colon K'_{x'} \to K_x$ in Proposition \ref{prop:alpha}.  We define
  the {\it theta lift of the orbit datum} $(\cO, x, \chi_x)$ to be
  $(\cO', x', \chi_{x'})$ where
  \[
  \chi_{x'} : =  \varsigma|_{\wtK'_{x'}} \otimes (\varsigma|_{\wtK_x}\otimes
  \chi_x)\circ \alpha.
  \]
  We write $\theta(\cO, x, \chi_x) = (\cO', x', \chi_{x'})$ which is
  well defined up to $\wtK_\bC$-conjugation.  Similarly we define the
  theta lift of $(\cO, x, [\chi_x])$ to be $\theta(\cO, x, [\chi_x])
  := (\cO', x', [\chi_{x'}])$.
\end{enumerate}
\end{definition}

\begin{thmA} \label{TC} Suppose $(G,G')$ is in the stable range where
  $G$ is the smaller member.  Let~$\rho$ be a genuine irreducible
  $(\fgg,\wtK)$-module. Suppose $\Set{(\cO_j, x_j, [\chi_{x_j}]) : j =
    1, \ldots, r }$ is the set of orbit data attached to $\rho^*$.
  Then $\Set{ \theta(\cO_j,x_j, [\chi_{x_j}]) : j = 1, \ldots, r }$ is
  the set of orbit data attached to $\Theta(\rho)$.
\end{thmA}

The next theorem is a corollary of Theorem \ref{TC}. 

\begin{thmA} \label{TD} Suppose $(G,G')$ is in the stable
  range where $G$ is the smaller member. Then
\[
\AV(\Theta(\rho)) = \theta(\AV(\rho^*))
\mbox{   and  } 
\AC(\Theta(\rho)) = \theta(\AC(\rho^*)).
\]
In particular if $\rho$ is unitarizable and excluding $(\dagger)$,
then $\AV(\theta(\rho)) = \theta(\AV(\rho^*))$ and $\AC(\theta(\rho))
= \theta(\AC(\rho^*))$ by Theorem \ref{TA}.
\end{thmA}

The proofs of Theorems \ref{TC} and \ref{TD} are given in Section
\ref{S44b}.  In these two theorems, we do not require that $\rho^*$ is
unitarizable. We will show in the proof of Lemma \ref{lem:cpxorb} that
the dimension of every $\theta(\overline{\cO_j})$ is equal to $\dim
\AV(\Theta(\rho))$, i.e. the Gelfand-Kirillov dimension
of~$\Theta(\rho)$. However there are examples where $\Theta(\rho)$ is
reducible and $\theta(\rho)$ has smaller Gelfand-Kirillov dimension
than that of $\Theta(\rho)$. In particular $\AV(\theta(\rho))$ does
not contain any~$\theta(\cO_j)$.

\smallskip

Theorem \ref{TD} overlaps with the previous work of \cite{NOT} and
\cite{Ya} where $G$ is a compact group. It also extends the work
\cite{NZ} where $\rho$ is a unitarizable lowest weight module.

\smallskip

We would like to relate a recent result of \cite{GZ} where Gomez and
Zhu show that the dimensions of the generalized Whittaker functionals
of the Casselman-Wallach globalizations of $\rho$ and $\Theta(\rho)$
are the same. It is a famous result of \cite{MW} that in the $p$-adic
case, the dimension of a space of generalized Whittaker functionals of
an algebraic irreducible representation is equal to the corresponding
multiplicity in its wavefront cycle. Theorem~\ref{TD} together with
\cite{GZ} could be interpreted as an evidence for the corresponding
phenomenon for real classical groups.

\subsection{}
Let $(\sfg, \sfK)$ and $\sfG$ as in Section \ref{S12}. For a
$(\sfg,\sfK)$-module $\varrho$ of finite length, we define
$\VC(\varrho)$ to be the complex variety cut out by the ideal
$\Gr(\Ann_{\cU(\sfg)} \varrho)$ in $\sfg^*$, where
$\Gr(\Ann_{\cU(\sfg)} \varrho)$ is the graded ideal of
$\Ann_{\cU(\fgg)} \rho$ in $\Gr \cU(\sfg) = \bC[\sfg^*]$.
Alternatively $\VC(\varrho)$ is the associated variety of the $(\sfg
\oplus \sfg, \Ad \sfG)$-module $\cU(\sfg)/\Ann_{\cU(\sfg)}
\varrho$. It is an $(\Ad^* \sfG)_\bC$-invariant complex variety
in $\fgg^*$ whose dimension is equal to $2 \dim \AV(\varrho)$.  By
Proposition~\ref{P24}, $\VC(\varrho^*) = \VC(\varrho)$.

\smallskip

We recall that $(G,G')$ is a type I irreducible dual pair in the stable range where $G$ is the smaller member. The actions of
$G$ and $G'$ on the symplectic manifold $W_\bR$ give two moment maps
(see \cite{DKP})
\begin{equation} \label{eq2a}
 \xymatrix{ \fgg^* & \ar[l]_{\phi_G \ \
    } W_\bR \otimes_\bR \bC \ar[r]^{\ \ \phi_{G'}} & \fgg'^*.}
\end{equation}
For an $\Ad^* G_\bC$-invariant complex subvariety $S$ of $\fgg^*$, we
define $\theta_\bC(S) = \phi_{G'}(\phi_G^{-1}(S))$. This is an $\Ad^*
G_\bC'$-invariant complex subvariety of $\fgg'^*$. We state a
corollary of Theorem~\ref{TD}.

\begin{corA} \label{CE} Suppose $(G,G')$ is in the stable range where
  $G$ is the smaller member. Let~$\rho$ be a genuine irreducible
  $(\fgg,\wtK)$-module. Then
\[
\VC(\Theta(\rho)) = \theta_\bC(\VC(\rho)).
\]
\end{corA}

The proof is given in Section \ref{S44}.

The above corollary overlaps with Theorem 0.9 in \cite{Pz} where
Przebinda proves the identity $\VC(\theta(\rho)) =
\theta_\bC(\VC(\rho))$ for dual pairs and unitarizable $\rho$
satisfying some technical conditions.

\subsection{} \label{S13} In Section \ref{S5}, we consider
representations whose $\wtK$-spectrums are the same as the global
sections of $\wtK_\bC$-equivariant algebraic vector bundles on
nilpotent orbits.  We will show that theta lifts in the stable
range preserve such property.
 
First we set up some notation. Let $\sfK$ be a compact group.  Let
$\cO$ be a $\sfK_\bC$-homogeneous space and let $x \in \cO$.  Let $\pi
: \sfK_\bC \rightarrow \cO$ be the natural quotient map given by
$\pi(k) = (\Ad^* k) x$.  Let $\sfK_x$ be the stabilizer of $x$ in
$\sfK_\bC$. For a rational $\sfK_x$-module $(\chi_x, V_x)$, we define
the $\sfK_\bC$-equivariant pre-sheaf $\sL$ on $\cO$ by $\sL(U) =
(\bC[\pi^{-1}(U)] \otimes_\bC V_x)^{\sfK_x}$ for all open subsets $U$
of~$\cO$. By \cite{CPS}, $\sL$ is a $\sfK_\bC$-equivariant
quasi-coherent sheaf with fiber $\chi_x$ at $x$. Moreover, by
Theorem~2.7 in \cite{CPS}, $\chi_x \leftrightarrow \sL$ gives an
equivalence of categories between the category of rational
representations of~$\sfK_x$ and the category of $\sfK_\bC$-equivariant
quasi-coherent sheaves on $\cO \simeq \sfK_\bC/\sfK_x$. We define the
$(\bC[\cO],\sfK_\bC)$-module
\[
\Ind_{\sfK_x}^{\sfK_\bC} \chi_x = (\bC[\sfK_\bC] \otimes V_x)^{\sfK_x} =
H^0(\cO,\sL) .
\]

\smallskip

If $(\cO, x, \chi_x)$ appears in the orbit data attached to a
filtration of a finite length $(\fgg,\wtK)$-module, then we have
\begin{equation}\label{eq:condchi}
  \sL \text{ is generated by its space of global sections } 
  \Ind_{\wtK_x}^{\wtK_\bC} \chi_x. 
\end{equation}
For the rest of this section we will assume that data $(\cO,x,\chi_x)$
satisfy \eqref{eq:condchi}.

We exclude following special cases:
\begin{equation}\label{eq:ddagger} 
\begin{split}
  (G,G') =& (\Sp(2n,\bR),\rO(p,q)) \text{ where } p = 2n \text{ or }
  q=2n;\\
  (G,G') = & (\Sp(2n,\bC), \rO(4n,\bC)).
\end{split} \tag{$\dagger \dagger$}
\end{equation}

\begin{thmA} \label{TF} Suppose $(G,G')$ is in the stable range where
  $G$ is the smaller member. We exclude the special
  case~\eqref{eq:ddagger} above. Let $\rho$ be an irreducible
  admissible genuine $(\fgg, \wtK)$-module.  Let $(\cO, x, \chi_x)$ be
  an orbit datum satisfying \eqref{eq:condchi} such that, as
  $\wtK$-modules
  \[
  \rho^* \ \simeq \ \Ind_{\wtK_x}^{\wtK_\bC}\chi_x.
  \]
  Let $(\cO',x',\chi_{x'})$ be the theta lifting of $(\cO,x,\chi_x)$.
  Then, as $\wtK'$-modules,
\[
\Theta(\rho) \ \simeq \ \Ind_{\wtK'_{x'}}^{\wtK'_\bC} \chi_{x'}.
\]
\end{thmA}

The proof is given in Section \ref{S53}.

\subsection{} \label{S16} We relate our results with a conjecture of
Vogan on geometric quantizations and unipotent representations.

\begin{definition}[Definition~7.13 in \cite{Vo89}]
  Let $\cO\in \frakN_{K_\bC}(\fpp^*)$ and $x\in \cO$. The stabilizer
  $K_x$ acts on the cotangent space $\rT^*_x\cO = (\fkk/\fkk_x)^*$.
  We define the character $\gamma_x$ of $K_x$ by
\[
\gamma_x(k) = \det(\Ad(k)|_{(\fkk/\fkk_x)^*}) \qquad \forall k \in K_x.
\]
A rational representation $\chi_x$ of the double cover $\wtK_x$ is
called {\it admissible} if
\begin{equation} \label{eq5}
  \chi_x(\exp(X)) = \gamma_x(\exp(X/2)) \cdot {\mathrm{Id}} 
  \qquad \forall X\in \fkk_x. 
\end{equation}
An orbit datum $(\cO, x, [\chi_x])$ is called an {\it admissible orbit
  datum} if $\chi_x$ is admissible.  An orbit $\cO\in
\frakN_{K_\bC}(\fpp^*)$ is called {\it admissible} if it is part of an
admissible datum. A representation $\chi_x$ of $\wtK_x$
satisfying \eqref{eq5} is uniquely determined by its character
$[\chi_x]$.
\end{definition}

A $(\fgg, \wtK)$-module $\rho$ is said to have
$\wtK$-spectrum determined by an admissible orbit datum $(\cO,
x,[\chi_x])$ if
\begin{equation} \label{eq:4} 
  \rho|_{\wtK} \, \simeq \, \Ind_{\wtK_x}^{\wtK_\bC} \chi_x
\end{equation}
as a $\wtK$-module. Such a representation $\rho$ could be
considered as a quantization of the orbit~$\cO$. In Conjecture 12.1 of
\cite{Vo89}, Vogan conjectured that, for every admissible orbit datum
$(\cO, x, [\chi_x])$ satisfying certain technical conditions and
$\partial \cO$ has codimension at least~$2$ in $\bcO$, one can attach
a unipotent representation $\rho$ to this orbit datum and
$\rho$ satisfies~\eqref{eq:4}.

\smallskip

In Section~\ref{sec:adm}, we will show that the notion of
admissibility is compatible with theta lifts in the stable range.

\begin{propA}\label{PG}
  Suppose $(G,G')$ is in the stable range where $G$ is the smaller
  member. Let $(\cO, x, [\chi_{x}])$ be an admissible orbit datum for
  $\wtG$. Then its theta lift $\theta(\cO, x, [\chi_{x}])$ is an
  admissible orbit datum for $\wtG'$.
\end{propA}

The above is a direct consequence of  Proposition \ref{prop:adm}.

Suppose $(G,G')$ is in the stable range where $G$ is the smaller
member and excluding the special case \eqref{eq:ddagger}. Let $\rho$
be an irreducible unitarizable $(\fgg, \wtK)$-module whose
$\wtK$-spectrum is given by some admissible orbit datum
$(\cO,x,[\chi_x])$.  It follows from Appendix \ref{SB1} that $\rho^*$
is an irreducible unitarizable $(\fgg, \wtK)$-module whose
$\wtK$-spectrum is given by the admissible orbit datum
\[
C(\cO,x,[\chi_x]) := (C(\cO),\Ad^*C(x),[\chi_x \circ C])
\] 
where $C$ is a dualizing automorphism on $\wtG$. By Theorem~\ref{TA},
Theorem~\ref{TF} and Proposition~\ref{PG}, $\theta(\rho)$ is an
irreducible unitarizable $(\fgg',\wtK')$-module whose $\wtK'$-spectrum
is given by the admissible orbit datum $\theta(C(\cO,x,[\chi_x]))$.

\subsection{} Finally we construct a series of candidates for
unipotent representations.  Let
\[
G_0, G_1, G_2, \cdots, G_n, \cdots
\]
be a sequence of real classical groups satisfying the following properties:
\begin{enumerate}[(i)]
\item each pair $(G_n, G_{n+1})$ is an irreducible type I reductive
  dual pair with $G_n$ being the smaller member excluding the special
  case \eqref{eq:ddagger}.

\item The corresponding double covers $\wtG_n$ of $G_n$ for the dual
  pairs $(G_{n-1},G_n)$ and $(G_n,G_{n+1})$ are isomorphic. We
   fix an isomorphism between these two double covers of
    $G_n$.

\item The covering group $\wtG_0$ has an irreducible genuine one
  dimensional unitary representation $\rho_0$ such that
  $\rho_0|_{\fgg_0}$ is trivial.
\end{enumerate}
It is clear that $\rho_0$ is attached to the admissible datum $( \Set{
  0 }, 0, {\rho_0}|_{(\wtK_0)_\bC})$.

Let $C_n$ be a dualizing automorphism on $\wtG_n$.  Starting from
$\rho_0$, we define inductively $\rho_{n+1} = \theta(\rho_n)$ and
$(\cO_{n+1}, x_{n+1}, \chi_{n+1}) = \theta(C_n(\cO_n, x_n,
\chi_n))$. The following theorem follows from Section \ref{S16}.

\begin{thmA}
  The $(\fgg_n, \wtK_n)$-module $\rho_n$ is an irreducible and
  unitarizable representation attached to the admissible orbit datum
  $(\cO_n,x_n,\chi_n)$.  Moreover, as $\wtK_n$-module,
  \[
 \rho_n \ \simeq \ \Ind_{\wtK_{x_n}}^{(\wtK_n)_\bC}\chi_n.
  \]
\end{thmA}

The above theorem generalizes a result of Yang \cite{Yn1} \cite{Yn2}
where he proves the above theorem for $\rho_1$. A related result on
Dixmier algebras is given in \cite{Br}.

\medskip

\subsection*{Acknowledgment}
We would like to thank Chen-bo Zhu, CheeWhye Chin, Roger Howe,
Jian-Shu Li, Kyo Nishiyama and Peter Trapa for enlightening
discussions.  We also thank the referees and C. Chin for
  careful reading and pointing out a number of inaccuracies in the
first draft. The first author is supported by a MOE-NUS grant
MOE2010-T2-2-113. The second author would like to thank the
hospitality of Aoyama Gakuin University and Hong Kong University of
Science and Technology where part of this paper is written.

\section{Theta lifts of unitary representations in the stable
  range} 

\subsection{} \label{S21} Let $(G,G')$ be a type I irreducible
reductive dual pair in $\Sp(W_\bR)$. We list them in
Table~\ref{tab:sstabncp} below.  We say it is in the stable range with
$G$ being the smaller member if it satisfies the conditions in the
last column of the table.
\begin{table}[hbtp]
  \centering
  \begin{tabular}[h]{c|cc|c}
    & $G$ & $G'$  & Stable range \\
    \hline
    Case~$\bR$ &
    $\Sp(2n,\bR)$ &  $\rO(p,q)$& $2n \leq p, q$  \\
    &$\rO(p,q)$ & $\Sp(2n,\bR)$ & $p+q \leq n$ \\
    \hline
    Case~$\bC$ & $\rU(n_1,n_2)$ & $\rU(p,q)$ &  $n_1+n_2 \leq p,q$ \\
    \hline
    Case~$\bH$ 
    & $\O^*(2n)$ & $\Sp(p,q)$ & $n \leq p,q$\\
    & $\Sp(p,q)$ & $\O^*(2n)$ & $2(p+q) \leq n$ \\    \hline
    Complex groups &
    $\Sp(2n,\bC)$ &  $\rO(p,\bC)$& $4n \leq p$ \\
    & $\rO(p,\bC)$ &  $\Sp(2n,\bC)$ &
    $p \leq n$
  \end{tabular}
\medskip
\caption{Stable range for irreducible Type I dual pairs}
  \label{tab:sstabncp}
\end{table}

We follow the notation in \cite{Howe89}. By Fact 1 in \cite{Howe89}, $K'$
is a member of a reductive dual pair $(K',M)$ in $\Sp(W_\bR)$. We form
the following see-saw pair in $\Sp(W_\bR)$:
\begin{equation} \label{eqdp}
\vcenter{
\xymatrix{
 G' \ar@{-}[d] \ar@{-}[dr] &
 M  \ar@{-}[d] \ar@{-}[dl]\\
 K' & G.}}
\end{equation}
The complex Lie algebra of $M$ has Cartan decomposition $\fmm =
\fmm^{(2,0)} \oplus \fmm^{(1,1)} \oplus \fmm^{(0,2)}$ where
$\fmm^{(1,1)}$ is the complexified Lie algebra of a maximal compact
subgroup $M^{(1,1)}$ of $M$.

Let $\widetilde{\cH} = \set{ v \in \sY | Xv = 0,\,\forall X \in
  \fmm^{(0,2)}}$ be the space of $\wtK'$-harmonics in $\sY$. As an
$\wtMoo \times \wtK'$-module,
\begin{equation} \label{eqharmonics} \widetilde{\cH} =
  \mbox{$\bigoplus_{\sigma' \in \widehat{\wtK'}}$} \ \sigma \otimes
  \sigma'
\end{equation}
where each $\sigma$ is either zero or an irreducible genuine 
$\wtMoo$-module uniquely determined by $\sigma'$.

\begin{prop} \label{P4} Suppose $(G,G')$ is in the stable range where $G$
  is the smaller member. Let~$\rho$ be an irreducible genuine
  $(\fgg,\wtK)$-module. Then as $\wtK'$-modules
\[
\Theta(\rho)|_{\wtK'} = \bigoplus_{\sigma' \in \widehat{\wtK'}}
m_{\sigma'} \sigma' \simeq (\widetilde{\cH} \otimes \rho^*|_{\wtK})^K
\] 
where $m_{\sigma'}$ is the multiplicity of $\sigma'$ in
$\Theta(\rho)$. We have $m_{\sigma'} = \dim \Hom_{\wtK}(\sigma,
\rho)$.
\end{prop}

If $\rho$ is the Harish-Chandra module of a discrete series
representation of $\wtG$, the above proposition is Corollary 5.3 in
\cite{Howe:1983reci}. 

%We also remark that all genuine $\wtK$-types
%appear in~$\widetilde{\cH}$. Hence the above proposition proves that
%$\Theta(\rho)$ is nonzero in the stable range.

\begin{proof}
  Let $L(\sigma)$ denote the (full) theta lift of $\sigma'$, which is
  a unitarizable lowest weight module of $\widetilde{M}$. The
  fact that the pair $(G,G')$ is in the stable range implies
that
  \begin{equation} \label{eq8} L(\sigma) = \cU(\fmm)
    \otimes_{\cU(\fmm^{(1,1)} \oplus \fmm^{(0,2)})} \sigma \simeq
    \cU(\fgg)\otimes_{\cU(\fkk)} \sigma
\end{equation}
as a $\fgg$-module. The first equality follows from Jantzen irreducibility criterion (see Section~6 in \cite{EHW}). In this case
  $L(\sigma)$ is a Harish-Chandra module of a (limit of) holomorphic discrete series representations (see Sections~II.8.2 and III.8.1 in
  \cite{KaV}). The second equality follows from $\fkk = \fgg\cap
(\fmm^{(1,1)}\oplus \fmm^{(0,2)})$ and $\fmm = (\fmm^{(1,1)}\oplus
\fmm^{(0,2)})+ \fgg$ (see (3.4) in \cite{Howe89}). Applying this to
the see-saw pair \eqref{eqdp}, we get
\begin{equation} \label{eq9} 
  m_{\sigma'} = \dim \Hom_{\wtK'}(\sigma', \Theta(\rho)) = 
\dim \Hom_{\fgg,\wtK}(L(\sigma), \rho) = \dim
  \Hom_{\wtK}(\sigma,\rho).
\end{equation}
This proves the proposition.
\end{proof}

\subsection{} \label{S23a}
Let $(\rho, V_\rho)$ be an irreducible unitarizable
Harish-Chandra module of $\wtG$. For the rest of this section we will
prove Theorem \ref{TA}.

First we recall Li's construction of $\theta(\rho)$ \cite{Li1989}. We
denote an element in the inverse image of $g\in M$ by $\tg\in
\wtM$. The actual choice of $\tg$ will not affect the
calculation. Define
\[
\inn{v_1\otimes w_1}{v_2\otimes w_2} = \int_{G}
\inn{\rho^*(\tg)v_1}{v_2}_{\rho^*}
\inn{\tg \cdot w_1}{w_2}_{\sY}  dg
\]
for all $v_1 \otimes w_1, v_2\otimes w_2 \in V_{\rho^*}
\otimes \sY$. All pairings are done in the completion of the
Harish-Chandra modules. 
We set 
\[
\Rad(\inn{}{}) = \Set{\Phi\in V_{\rho^*} \otimes \sY |
  \inn{\Phi}{\Phi'} = 0, \forall \Phi'\in V_{\rho^*} \otimes
  \sY}.
\]
Let
\[
H = (V_{\rho^*} \otimes \sY) /\Rad(\inn{}{}).
\]
We claim that $H \simeq \theta(\rho)$ as irreducible unitarizable
Harish-Chandra modules of $G'$. \linebreak Indeed Li \cite{Li1989}
uses smooth vectors in the definition of $\inn{}{}$
and likewise defines \linebreak $H^\infty = ((V_{\rho^*})^\infty
\otimes \sY^\infty)/ \Rad(\inn{}{}^\infty)$. Theorem 6.1 in
\cite{Li1989} shows that $\theta(\rho)$ is the Harish-Chandra module
of $H^\infty$.  Since $H$ is $\wtK'$-finite and dense in $H^\infty$,
it is equal to the Harish-Chandra module $\theta(\rho)$ of $H^\infty$.
This proves our claim.

We refer to $(\sigma,V_\sigma)$ in $\eqref{eqharmonics}$ and
$L(\sigma)$ in \eqref{eq8}. Then $L(\sigma)$ is an irreducible
unitarizable Harish-Chandra module of $\wtM$ and $\sY =
\bigoplus_{\sigma'} L(\sigma) \otimes \sigma'$.

We set
\[
\inn{v_1\otimes w_1}{v_2\otimes w_2}_{\rho^*}^{\sigma'} = \int_G
\inn{\rho^*(\tg)v_1}{v_2}_{\rho^*} \inn{\tg \cdot
  w_1}{w_2}_{L(\sigma)} \,dg \quad \forall v_i\otimes w_i \in
V_{\rho^*} \otimes L(\sigma)
\]
and define
\[ 
H(\sigma') = \left(V_{\rho^*} \otimes L(\sigma) \right)/
\Rad(\inn{}{}_{\rho^*}^{\sigma'}).
\]
Now $\dim H(\sigma')$ is the multiplicity of $\sigma'$ in $H$ and we
have
\[
H = \mbox{$\bigoplus_{\sigma'}$} H(\sigma') \otimes \sigma'.
\]

 We consider following embeddings:
\begin{equation}\label{eq:multembedding1}
  \xymatrix@R=1em{H(\sigma') = (V_{\rho^*} \otimes L(\sigma))/\Rad(\inn{}{}_{\rho^*}^{\sigma'})
    \, \ar@{^(->}[r]^<>(.5){\iota} &
    \Hom_{G}( V_{\rho^*}^\infty \otimes L(\sigma)^\infty,\bC) \,  {}
    \ar@{^(->}[r]^{\ \ \mathrm{rest.}} & \Hom_{\fgg,K}(V_{\rho^*} \otimes L(\sigma),\bC)}
\end{equation}
where $\iota(\Phi)$ is given by
\begin{equation} \label{eq12}
\Phi \mapsto (\Phi'\mapsto \inn{\Phi'}{\Phi}_{\rho^*}^{\sigma'})\quad
\forall \Phi' \in V_{\rho^*}^\infty \otimes L(\sigma)^\infty. 
\end{equation}

The last term on the right hand side of
\eqref{eq:multembedding1} is
\begin{eqnarray*}
  \lefteqn{\Hom_{\fgg,K}(V_{\rho^*} \otimes L(\sigma) ,\bC) 
    = \Hom_{\fgg,\wtK}(L(\sigma),
    \Hom_\bC(V_{\rho^*},\bC))} \\
  & = & \Hom_{\fgg,\wtK}(L(\sigma),V_{\rho})\\
  & &\text{($L(\sigma)$ is
    $\wtK$-finite, so its image is in  
$\Hom_\bC(V_{\rho^*},\bC)_{\wtK-{\mathrm{finite}}} =  V_{\rho}$)}\\
    &=& 
    \Hom_{\wtK}(V_\sigma, V_\rho) \qquad \text{(by \eqref{eq8})}\\
    &=& \Hom_K(V_{\rho^*} \otimes V_\sigma,\bC) = 
    \Hom_\bC((V_{\rho^*} \otimes V_\sigma)^{K},\bC).
\end{eqnarray*}
The isomorphism between the first term and the last term in above
equalities is given by restriction. Combining these with
  \eqref{eq:multembedding1} gives an inclusion map
\begin{equation} \label{eqn12}
H(\sigma') \hookrightarrow \Hom_\bC((V_{\rho^*} \otimes
V_\sigma)^{K},\bC)
\end{equation}
given by \eqref{eq12} but for $\Phi' \in (V_{\rho^*} \otimes
V_\sigma)^{K}$.  By \eqref{eq9}
\[
\dim H(\sigma') \leq \dim \Hom_{\wtK}(V_\sigma, V_{\rho}) = \dim
\Hom_{\wtK'}(\sigma', \Theta(\rho))
\]
and it is finite. 

\begin{lemma} \label{L22} Let $\bfdd \in (V_{\rho^*} \otimes
  V_\sigma)^{K}$ be a nonzero vector. Then the pairing between
  $V_{\rho^*} \otimes L(\sigma)$ and $\bfdd$ in~\eqref{eq12} is
  non-vanishing.
\end{lemma}

The above lemma implies that \eqref{eqn12} is an isomorphism so that
$\Theta(\rho)$ and $H$ have the same $\wtK$-multiplicities.  This will
prove Theorem~\ref{TA}.

\medskip

In order to prove Lemma \ref{L22}, we first exhibit a globalization of the Harish-Chandra module $L(\sigma)$. Our references
are \cite{KaV} and~\cite{JaV}.  We refer to $M$ in \eqref{eqdp}. Let
\[
\Hol(\wtM,\wtMone,V_\sigma) = \Set{ f : \wtM \rightarrow V_\sigma |
\begin{array}{l}
  f \mbox{ is analytic,} \\
  f(\tg\tk) = \sigma(\tk^{-1}) f(\tg) \ \ \forall \tg \in \wtM, \tk 
\in \wtMone, \\
  r(X) f = 0 \ \ \forall X \in \fmm^{(0,2)}.
\end{array}
}.
\]
Here $r(X)$ denote the right derivation action.  Let $\Set{v_i}$ be an
orthonormal basis of $V_\sigma \subset L(\sigma)$.  Then
\[
\xi : v\mapsto \left(\tg\mapsto \sum_{i}\inn{\tg^{-1} \cdot
v}{v_i}_{L(\sigma)} v_i\right)
\]
defines an injective $(\fmm, \wtMone)$-module homomorphism $\xi \colon
L(\sigma) \to \Hol(\wtM, \wtMone;V_\sigma)$.

For any $g\in M$, there are unique elements $z(g) \in \fmm^{(0,2)}$,
$k(g) \in M^{(1,1)}_\bC$ and $z'(g) \in \fmm^{(2,0)}$ such that $g =
\exp(z(g))k(g) \exp(z'(g))$. The map $g \mapsto k(g)$ lifts to a map
$\tilde{k} : \wtM \rightarrow \wtMone_{\bC}$ (see page 8 in
\cite{KaV}).  Let $\Omega$ denote the image of the map~$z$ and let
$\zeta : \wtM \rightarrow M \stackrel{z}{\rightarrow} \Omega$ denote
the composite map. Then
\begin{equation}\label{eq:domain}
\Omega = \{ z(g) \in \fmm^{(0,2)} : g \in M \} \simeq
M/M^{(1,1)} \simeq \wtM/\wtM^{(1,1)}
\end{equation}
is a bounded symmetric domain in
$\fmm^{(0,2)}$ and
\[
M \subset \exp(\Omega) \cdot M^{(1,1)}_\bC \cdot \exp(\fmm^{(2,0)}).
\]

Let $\Hol(\Omega,V_\sigma)$ denote the space of holomorphic functions
on $\Omega$ with values in $V_\sigma$.  We define $\rP :
\Hol(\wtM,\wtM^{(1,1)},V_\sigma) \rightarrow \Hol(\Omega,V_\sigma)$ in
the following way: For $f \in \Hol(\wtM,\wtM^{(1,1)},V_\sigma)$, we
set $\rP f \in \Hol(\Omega,V_\sigma)$ by $\rP f(\tg\wtM^{(1,1)}) = \sigma(\tilde{k}(\tilde{g})) f(\tg)$. Then $\rP$ is a bijection using
\eqref{eq:domain}.

Let $\barxi = \rP \circ \xi : L(\sigma) \to \Hol(\Omega,
V_\sigma)$.  Let $\bC[\fmm^{(0,2)}]$ denote the space of polynomials
on~$\fmm^{(0,2)}$. Then $\barxi (L(\sigma))$ is the linear span of
\[
\Set{p \times \barxi(v) | p \in \bC[\fmm^{(0,2)}], v\in
  V_\sigma}
\] 
because $L(\sigma)$ is a full generalized Verma module.

\medskip

We write $V_\sigma = \bigoplus_{l \in L} D_l$ and $V_{\rho^*} =
\bigoplus_{j \in J} D_j$ as direct sums of irreducible $\wtK$-modules.
Then
\[
(V_{\rho^*} \otimes V_\sigma)^K =
 \bigoplus_{j \in J} \bigoplus_{l \in L} (D_j \otimes D_l)^K  
= \bigoplus_{l \in L} \bigoplus_{D_j \simeq D_l^*} (D_j \otimes D_l)^K.
\] 
Let $\{ d_{l \lambda} : \lambda = 1, \ldots, \dim D_l \}$ be
  an orthonormal basis of $D_l$ and let $\{ d_{j \lambda}^* \}$ be a
basis of $D_j \simeq D_l^*$ which is dual to $\{ d_{l \lambda} \}$.  Then a vector $\bfdd \in (V_{\rho^*}
  \otimes V_\sigma)^{K}$ in Lemma \ref{L22} is of the form
\begin{equation} \label{eqbfd}
\bfdd = \sum_{j \in J, l \in L} c_{jl} \left(\sum_{\lambda}
  d_{j\lambda}^* \otimes d_{l\lambda}\right)
\end{equation}
where $c_{jl} \in \bC$. Here $c_{jl} = 0$ unless $D_j \simeq
D_l^*$. We suppose $c_{j_0 l_0} \neq 0$ for some $j_0 \in J$ and $l_0
\in L$. 

Let
\[
\cC(\tG,\wtK;V_\sigma) = \Set{f \in \cC(\tG,V_\sigma) |
  f(\tg\tk) = \sigma(\tk^{-1}) f(\tg) \ \forall k \in \wtK}
\]
be the space of continuous sections. We define a $G$-module
  homomorphism $\xi_{\bfdd} \colon V_{\rho^*}^\infty \to \cC(\tG,\wtK;V_\sigma)$ by
\[
\xi_{\bfdd} : v \mapsto \left(\tg \mapsto \sum_{j \in J, l \in L} c_{jl} \sum_\lambda
  \overline{\inn{\rho^*(\tg^{-1})v}{d_{j\lambda}^*}_{\rho^*}}d_{l
    \lambda} \right) \ \ (\forall \tg \in \tG).
\]

Let $\Omega_0 \simeq \tG/\wtK$ denote the image $\zeta(\tG)$ in
$\Omega$. We have $\rP_0\colon \cC(\wtG,\wtK;V_\sigma) \to
\cC(\Omega_0;V_\sigma)$ defined by the same formula as $\rP$. We
denote $\barxi_{\bfdd} = \rP_0\circ \xi_{\bfdd}$. Let $dx$ be the
$G$-invariant measure on $\Omega_0 \simeq G/K$ compatible with the
Haar measure on $G$ i.e. $dg = dx \, dk$. We recall that $c_{j_0 l_0}
\neq 0$. Let $d_{l_0 \lambda_0}$ be a unit vector in the orthonormal
basis of $D_{l_0}$.  Let $w$ in $L(\sigma)$ such that
$\barxi(w)=p\times \barxi(d_{l_0 \lambda_0})$.  Let $v = d_{j_0
  \lambda_0}^*$ be the corresponding unit vector in the orthonormal
basis of $D_{j_0} \subseteq V_{\rho^*}$.  Then $v\otimes w \in
V_{\rho^*} \otimes L(\sigma)$ and we have
\begin{eqnarray}
  \inn{v\otimes w}{\bfdd}_{\rho^*}^{\sigma}
  &  =& \sum_{j \in J, l \in L} c_{jl} \int_{G}\sum_{\lambda} \inn{\rho^*(\tg^{-1})v}{d_{j\lambda}^*}_{\rho^*} \inn{\tg^{-1}\cdot w}{d_{l\lambda}}_{L(\sigma)}  \, dg \nonumber \\
  & = & \int_{G} \inn{\xi(w)(\tg)}{\xi_{\bfdd}(v)(\tg)}_{V_\sigma}\, dg \nonumber \\
  & = & \int_{G} \inn{\sigma(\widetilde{k(g)})\xi(w)(\tg)}{\sigma(\widetilde{k(g)})\xi_{\bfdd}(v)(\tg)}_{V_\sigma}\, dg \nonumber \\
  & = & \int_{G/K} \inn{\barxi(w)(gK)}{\barxi_{\bfdd}(v)(gK)}_{V_\sigma}\,dgK \nonumber\\
  & = & \int_{\Omega_0} p(x) \inn{\barxi(d_{l_0 \lambda_0})(x)}{\barxi_{\bfdd}(d_{j_0 \lambda_0}^*)(x)}_{V_\sigma}\,dx \nonumber \\
  & = & \int_{\Omega_0} p(x) f(x) dx  \label{eq14}
\end{eqnarray}
where $f(x) =
\inn{\barxi(d_{l_0 \lambda_0})(x)}{\barxi_{\bfdd}(d_{j_0 \lambda_0}^*)(x)}_{V_\sigma}$. The
function $f(x)$ is a nonzero continuous function because $f(0) =
\sum_{j \in J, l \in L} c_{jl} \sum_{\lambda}
\inn{d_{j_0\lambda_0}^*}{d_{j\lambda}^*}_{\rho^*} \inn{d_{l_0
    \lambda_0}}{d_{l\lambda}}_{L(\sigma)} = c_{j_0 l_0} \neq 0$.  We
extend $f(x)$ to the boundary of $\Omega_0$ by $0$.

By Li \cite{Li1989}, the integration \eqref{eq14} is absolutely
convergent for every $p\in \bC[\fmm^{(0,2)}]$.  This is the place
where we exclude the Case $(\dagger)$ in Section \ref{S11}.

\smallskip

It remains to show that \eqref{eq14} is nonzero for some $p(x) \in
\bC[\fmm^{(0,2)}]$.  By \cite{Howe:1983reci}, the restriction of
$\bC[\fmm^{(0,2)}]$ to the compact subset $\overline{\Omega_0}$ forms
a dense subset in $\cC(\overline{\Omega_0})$ under sup-norm by the
Stone-Weierstrass Theorem. Note that any open subset of
$\overline{\Omega_0}$ has non-zero measure. Hence
$\int_{\overline{\Omega_0}}p(x) f(x)\, dx$ is non-zero for some~$p(x)$
by an approximation of identity argument. This completes the proof of
Lemma \ref{L22} and Theorem~\ref{TA}.

\section{Natural filtrations and corresponding $(\cS(\fpp),K)$-modules}

\subsection{} 
Let $(G,G')$ be an irreducible type I dual pair as in Table
\ref{tab:sstabncp}. We do not assume that it is in the stable range except
in Lemma \ref{L3}. Let $\rho$ be an irreducible genuine
$(\fgg,\wtK)$-module. Let $\rho^*$ denote its dual (contragredient)
$(\fgg,\wtK)$-module and let $\rho' = \Theta(\rho) $ denote its full
theta lift. For any module $\varrho$, we denote its underlying space
by $V_\varrho$.

\subsection{} 
The Fock model $\sY$ is realized as complex polynomials on $W$, so
$\sY = \bigcup_b \sY_b$ is filtered by degrees. See Appendix
\ref{sec:Fock}. Let $(\tau, V_\tau)$ be a lowest degree $\wtK$-type of
$(\rho, V_\rho)$ with degree $j_0$.  Let $V_\tau \otimes V_{\tau'}$ be
the image of joint harmonics in $V_\rho\otimes V_{\rho'}$.  By
\cite{Howe89}, $V_{\rho'} = \cU(\fgg')V_{\tau'}$. Thus we define a
good filtration on $V_{\rho'} = \bigcup_j V'_j$ by setting
$V_j' = \cU_j(\fgg') V_{\tau'}$.

We view $V_{\rho^*} = \Hom_\bC(V_\rho,\bC)_{\wtK-{\mathrm{finite}}}$. Let
$V_{\tau^*}\subset V_{\rho^*}$ be an irreducible $\wtK$-submodule with
type $\tau^*$ which pairs perfectly with $V_\tau$.  By Theorem 13 (5)
in \cite{He:2000}, the lowest degree $\wtK$-type has multiplicity one
in $\rho$.  Hence $V_\tau$ and $V_{\tau^*}$ are well defined.
 
Likewise we define filtrations on $V_\rho$ and $V_{\rho^*}$ by
$\Set{V_j := \cU_j(\fgg) V_{\tau}}_{j \in \bN}$ and \linebreak $\Set{ V_j^* :=
  \cU_j(\fgg)V_{\tau^*}}_{j \in \bN}$ respectively. We will clarify
the relationships between them in Appendix \ref{SB2}.

\def\sspan{\mathrm{Span}}

\medskip 

Let $\bfE = V_{\rho^*} \otimes \sY$. We set $\bfE_{\fgg,K} =
\bfE/\sspan\set{ X v,kv -v| v \in \bfE, X \in \fgg, k \in K}$ and
$\bfE_{\fpp} = \bfE/ \sspan\set{ X' v | v \in \bfE, X'\in \fpp}$.  By
Proposition~2.3 in \cite{LMT},
\[
\Theta(\rho)  \simeq \bfE_{\fgg,K} =
\left(\bfE_{\fpp}\right)^{K}.
\]

Let $\eta\colon \bfE \twoheadrightarrow \left(\bfE_{\fpp}\right)^{K}
\simeq V_{\rho'}$ be the natural quotient map.  We define a filtration
on $\bfE$ by
\[
\bfE_j = \sum_{2a+b=j} V_a^* \otimes \sY_b.
\]

\begin{lemma}[Section~2~\cite{LMT}]
  We have $\eta(\bfE_{j_0+2j}) = \eta(\bfE_{2j_0+2j+1}) = V_j'$.  \qed
\end{lemma}

Let $\pr_\fpp$ and $\pr_K$ be the projection to $\fpp$-coinvariants
and $K$-invariants respectively.  The last lemma says that $\bfE_j$ is
compatible with the filtration $\{ V_j' \}_{j \in \bN}$ on
$V_{\rho'}$.  Taking the graded module, $\eta$ induces a map
\begin{equation} \label{eq19b} \xymatrix@C=3em{ \Gr{V_{\rho^*}}\otimes
    \Gr \sY\ar@{->>}[r]^{\epsilon \otimes 1 \ \ } & \Gr V_{\rho^*}
    \otimes_{\cS(\fpp)} \Gr \sY \ar@{->>}[r] & \Gr(\pr_\fpp(\bfE))
    \ar@{->>}[r]^<>(.5){\Gr\pr_K} & \Gr V_{\rho'}.}
\end{equation}
Here $\epsilon : \Gr{V_{\rho^*}} \rightarrow \Gr{V_{\rho^*}}$ is the
$(\cS(\fpp),\wtK)$-module isomorphism such that $\epsilon(x) = (-1)^a
x$ for all $x \in \Gr^a{V_{\rho^*}}$.

\subsection{}
We recall that $\rU$ is a maximal compact subgroup of
$\Sp(W_\bR)$. Let $\sp^{(1,1)}$ be the complexified Lie algebra of
$\rU$. Let $\fsp(W_\bR) \otimes \bC = \sp^{(2,0)} \oplus \sp^{(1,1)}
\oplus \sp^{(0,2)}$ denote the complexified Cartan decomposition (see
Section \ref{SA1}).  Let $\fss = \sp^{(2,0)} \oplus \sp^{(0,2)}$. We
recall that $\varsigma$ is the minimal one dimensional
$\widetilde{\rU}$-type of the Fock model~$\sY$. We
extend $\varsigma$ to an $(\cS(\fss),\widetilde{\rU})$-module where
$\fss$ acts trivially. We will continue to denote this one dimensional
module by $\varsigma$.  In this way, $\Gr \sY =
\bigoplus(\sY_{a+1}/\sY_a) \simeq \varsigma \otimes \bC[W]$. Here
$\rU$ acts on $\bC[W]$ by $(k \cdot f)(w) = f(k^{-1} w)$ for $k \in
\rU$, $f \in \bC[W]$ and $w \in W$ (c.f. Section \ref{sec:Fock}). The
algebra $\cS( \sp^{(0,2)})$ acts trivially on $\bC[W]$ while $\cS(
\sp^{(2,0)})$ acts by multiplication by degree two homogeneous
polynomials. Since $(G,G')$ is a reductive dual pair in $\rSp(W_\bR)$,
we denote the restriction of $\varsigma$ as a
$(\cS(\fpp),\wtK)$-module by $\varsigma|_{\wtK}$. Similarly we get a
one dimensional $(\cS(\fpp'),\wtK')$-module $\varsigma|_{\wtK'}$.

Let $\bfA = \varsigma|_{\wtK}\otimes \Gr V_{\rho^*}$ and $\bfB =
\varsigma|_{\wtK'}^{-1}\otimes \Gr V_{\rho'}$. Since $\rho$ is a genuine
Harish-Chandra module of~$\tG$, $\bfA$ is an
$(\cS(\fpp),K_\bC)$-module. Similarly $\bfB$ is an
$(\cS(\fpp'),K_\bC')$-module. 

We note that that $K_\bC$ acts on $\bfA \otimes \bC[W]$ reductively
and preserves the degrees. Then~\eqref{eq19b} gives the following
$(\cS(\fpp'),K_\bC')$-module morphisms
\begin{equation} \label{eq:gretaiso} \xymatrix{ \bfA
    \otimes_{\cS(\fpp)} \bC[W] \ar@{->>}[r] &
    \left(\bfA \otimes_{\cS(\fpp)} \bC[W] \right)^{K_\bC}
    \ar@{->>}[r]^<>(.5){\eta_0}& \varsigma|_{\wtK'}^{-1} \otimes
    \left( \Gr(\pr_\fpp(\bfE))\right)^{K_\bC} \simeq \bfB.  }
\end{equation}
The merit of introducing $\varsigma$ is that the $\wtK \cdot
\wtK'$ action on $\Gr \sY$ descends to a geometric $K_\bC \cdot
K_\bC'$ action on $\bC[W]$.

Since $\bC[W]$ is an $(\cS(\fss), \rU_\bC)$-module, it is also an $(\cS(\fpp), K_\bC) \times (\cS(\fpp'), K_\bC')$-module.

\begin{lemma} \label{L3} Suppose $(G,G')$ is in the stable range where
  $G$ is the smaller member.  Then~$\eta_0$ in
  \eqref{eq:gretaiso} is an isomorphism, i.e.
\[
\bfB \simeq (\bfA \otimes_{\cS(\fpp)} \bC[W])^{K_\bC}
\]
as $(\cS(\fpp'), K_\bC')$-modules. Here $(\cS(\fpp'), K_\bC')$ acts trivially on $\bfA$. 
\end{lemma}

\begin{proof}
  We recall that $\widetilde{\cH}$ denotes the space of harmonics in
  \eqref{eqharmonics}. Let $\cH = \varsigma^{-1}
  \widetilde{\cH}$. Under the stable range assumption $\bC[W] =
  \cS(\fpp) \otimes \cH$ as an $(\cS(\fpp), K_\bC) \times
  K_\bC'$-module.  Since $\bfA = \varsigma|_{\wtK} \otimes
  V_{\rho^*}$, as $K_\bC'$-modules,
\[
(\bfA \otimes_{\cS(\fpp)} \bC[W])^{K_\bC} =(\bfA \otimes_{\cS(\fpp)}
(\cS(\fpp) \otimes \cH))^{K_\bC} = (\bfA|_{K_\bC} \otimes \cH)^{K_\bC}
\simeq \bfB|_{K_\bC'}
\]
by Proposition \ref{P4}. The map $\eta_0$ is a surjection and $\bfB$
is an admissible $K_\bC'$-module. The lemma follows from the
equality of $K'$-types.
\end{proof}

\subsection{} \label{S33} {\it Proof of Theorem \ref{TB}.}  Since the
filtration on $\rho^*$ (resp. $\rho'$) is good, the graded module
$\bfA$ (resp. $\bfB$) is a finitely generated $\cS(\fpp)$-module
(resp. $\cS(\fpp')$-module). Let $\sA$ be the associated coherent
sheaf of $\bfA$ on $\fpp^*$.  Using the moment maps
\[
\xymatrix{\fpp^* & \ar[l]_{\phi} W \ar[r]^{\phi'} & \fpp'^*}
\] 
we see that the associated quasi-coherent sheaf of $\bfA
\otimes_{\cS(\fpp)} \bC[W]$ on $\fpp'^*$ is $\phi_*' \phi^* \sA $. Let
$\sB $ be the associated quasi-coherent sheaf of $\bfB$ on $\fpp'^*$.

By definition, $\AV(\rho^*) = \Supp(\sA)$ and $\AV(\Theta(\rho)) =
\Supp(\sB)$.  By \eqref{eq:gretaiso}, $\sB $ is a subquotient of the
quasi-coherent sheaf $\phi_*'\phi^* \sA$ so
  \[ \Supp(\sB ) \subseteq \Supp(\phi_*'
  \phi^*\sA ) \subseteq\overline{\phi'(\Supp(\phi^*\sA ))}
  \subseteq \overline{\phi'(\phi^{-1}(\Supp(\sA)) )} = \theta(\Supp(\sA)).
  \] 
This proves the theorem. \qed

\smallskip

The above proof also applies to type II reductive dual pairs.

\section{Associated cycles}

\subsection{} 
Throughout this section, we suppose $(G,G')$ is in the stable range
where $G$ is the smaller member. Let $\rho$ be an irreducible genuine
$(\fgg,\wtK)$-module. The objective of this section is to prove
Theorems \ref{TC} and \ref{TD}.

%Let $\cO \in \frakN_{K_\bC}(\fpp^*)$ and $\cO'=\theta(\cO)$.  Pick a
%$w \in W$ such that $x = \phi(w) \in \cO$ and $x' = \phi'(w) \in
%\cO'$. 

\begin{prop} \label{prop:alpha} Suppose $(G,G')$ is in the stable
  range where $G$ is the smaller member. Let $\cO \in
  \frakN_{K_\bC}(\fpp^*)$ and $\cO' = \theta(\cO)$.  We fix a $w\in W$
  such that $x = \phi(w)\in \cO$ and $x'= \phi'(w) \in \cO'$. Let $K_x
  = \Stab_{K_\bC}(x)$ and $K'_{x'} =\Stab_{K'_\bC}(x')$.
\begin{enumerate}[(i)]
\item For every $k' \in K'_{x'}$, there exists a unique $k \in K_{x}$
  such that $(k')^{-1} \cdot w = k\cdot w$. We denote $k$ by
  $\alpha(k')$.

\item The function $\alpha : K'_{x'}\twoheadrightarrow K_{x}$ defined
  by $k' \mapsto \alpha(k')$ in (i) is a surjective group
  homomorphism. In particular,
\begin{equation}\label{eq:stab}
\Stab_{K_\bC\times K'_\bC}(w) = K_x\times_{\alpha} K'_{x'} :=
\Set{(\alpha(k'),k')|k'\in K'_{x'}}.
\end{equation}
\end{enumerate}
\end{prop}

We will prove this proposition in Section \ref{SA4} after we study
some properties of the moment maps. The group homomorphism $\alpha$
depends on the choice of~$w$. Indeed if we replace $w$ by
$(k_0,k'_0)\cdot w$, then the corresponding group homomorphism
becomes $\tilde{\alpha} \colon K'_{k'_0\cdot x'} \to K_{k_0\cdot x}$
which is given by $k' \mapsto k_0\alpha({k'_0}^{-1}k' k'_0)k_0^{-1}$.

Pre-composition with $\alpha$ defines a map
$\alpha^*$ from the set of $K_{x}$-modules (resp. virtual characters
of $K_{x}$) to the set of $K_{x'}'$-modules (resp. virtual characters
of $K_{x'}'$).

\subsection{} 
The next result is a key lemma which could be viewed as an enhancement of Section~1.3 in \cite{NZ}.  One may skip its proof in
the first reading.

\begin{lemma} \label{lem:isotropy1} Let $A$ be a $(\bC[\bcO],
  K_\bC)$-module. Define an $(\cS(\fpp'),K'_\bC)$-module by
\[B = (\bC[W]\otimes_{\cS(\fpp)} A)^{K_\bC}.\] Then $B$ is a
$(\bC[\bcOp],K'_\bC)$-module.  

Let $\sA $ and $\sB $ be the quasi-coherent sheaves on $\bcO$ and
$\bcOp$ associated to $A$ and $B$ respectively\footnote{We will abuse
  notation and continue to denote their extensions by zero to $\fpp^*$
  and $\fpp'^*$ by $\sA$ and $\sB$ respectively.}. Then we have the
following isomorphism of $K'_{x'}$-modules:
\[
i_{x'}^*\sB \simeq \alpha^*(i_{x}^* \sA).
\]
In particular, $\dim i_{x'}^*\sB = \dim i_x^*\sA$ if $A$ is finitely
generated.
\end{lemma}

\begin{proof}
  See \cite{Ma:thesis}. Let $Z = \phi^{-1}(\overline{\cO})$ be the set
  theoretical inverse image of $\overline{\cO}$.  We consider following diagram
\[
\xymatrix{
\set{x} \ar@{^{(}->}[dd]^{i_{x}} & 
\set{w} \ar[l]_{\simeq} \ar@{^{(}->}[d]^{i_w} \ar[dr]^{\simeq} &   \\
& Z_{x'}\ar@{^{(}->}[d]^{i_{Z_{x'}}} \ar[r] 
& \set{x'} \ar@{^{(}->}[d]^{i_{x'}} \\
\overline{\cO} \ar@{^{(}->}[d]^{i_{\overline{\cO}}} & 
Z \ar@{->>}[l]_{\phi|_Z} \ar[r]^{\phi'|_Z} \ar@{^{(}->}[d]^{i_Z} & 
\overline{\cO'}\ar@{^{(}->}[d]^{i_{\overline{\cO'}}}\\
  \fpp^* & W \ar[l]^{\phi} \ar[r]_{\phi'}&\fpp'^* }
\]
By Lemma~\ref{lem:rednorm}, the scheme theoretical inverse image $W
\times_{\fpp^*}\overline{\cO}$ is reduced, i.e.  $\bC[Z] = \bC[W]
\otimes_{S(\fpp )} \bC[\bcO]$. Then
\[
B = (\bC[W] \otimes_{\cS(\fpp )} A)^{K_\bC} =
(\bC[W] \otimes_{S(\fpp )} \bC[\bcO] \otimes_{\bC[\bcO]} A)^{K_\bC}
= (\bC[Z] \otimes_{\bC[\overline{\cO}]} A)^{K_\bC}
\]
as an $(\cS(\fpp'),K_\bC')$-module.  By Lemma~\ref{L21}, $\bC[\bcOp] =
\bC[Z]^{K_\bC}$ so $B$ is a $\bC[\bcOp]$-module.

We recall that $x'$ is a point in $\cO'$.  Let $Z_{x'} = Z
\times_{\bcOp} \set{x'}$ be the scheme theoretical fiber.  Since
$\phi'|_Z\colon Z\to \bcOp$ is dominant and we are in characteristic
zero, $Z_{x'}$ is reduced.  Let $m(x')$ be the maximal ideal in
$\cS(\fpp')$ corresponding to the point $x'$.

Since taking $K_\bC$-invariant is an exact functor and
$\phi'^*(\cS(\fpp'))$ is $K_\bC$-invariant, we have
\begin{eqnarray} 
  i_{x'}^*\sB  &=&(\cS(\fpp')/m(x')) \otimes_{\cS(\fpp')} 
  (\bC[Z]\otimes_{\bC[\overline{\cO}]} A)^{K_\bC} \nonumber \\
  & = & 
  (\bC[Z]\otimes_{\bC[\overline{\cO}]} A)^{K_\bC}/
  (m(x')\bC[Z]\otimes_{\bC[\overline{\cO}]} A)^{K_\bC} \nonumber \\
  & = &
  \left(\bC[Z]\otimes_{\bC[\overline{\cO}]} A)/
    (m(x') \bC[Z]\otimes_{\bC[\overline{\cO}]} A) \right)^{K_\bC} \nonumber \\
  & = & \left( (\bC[Z]/m(x')\bC[Z]) \otimes_{\bC[\overline{\cO}]} A
\right)^{K_\bC} \nonumber \\
  & = &(\bC[Z_{x'}]\otimes_{\bC[\overline{\cO}]} A)^{K_\bC}. \label{eq15}
\end{eqnarray}

Let $\sZ := i_{Z_{x'}}^*(\phi|_Z)^* \sA$. Then $\sZ(Z_{x'}) =
\bC[Z_{x'}]\otimes_{\bC[\overline{\cO}]} A$.  By
Lemma~\ref{lem:fullrank} (ii), $Z_{x'}$ is a $K_\bC \times K_{x'}$-orbit
generated by $w$. Let $S_w = \Stab_{K_\bC \times K'_{x'}}(w)$. By
\eqref{eq:stab}, $S_w = K_x \times_\alpha K'_{x'} = \{ (\alpha(k'),k')
\in K_x\times K'_{x'} \}$.  Then by Theorem~2.7 in \cite{CPS},
\[
\sZ(Z_{x'})=
\Ind_{K_x\times_\alpha K'_{x'}}^{K_\bC \times K'_{x'}} \chi
\]
where $\chi$ is the fiber of $\sZ$ at $w$. By the above commutative
diagram, we have $S_w$-module isomorphisms
\[
\chi = i_w^*\sZ = i_w^* i_{Z_{x'}}^* (\phi|_Z)^* \sA  \simeq i_{x}^*\sA
\]
where $(\alpha(k'), k') \in S_w = K_x \times_\alpha K_{x'}'$ acts on
$i_x^*\sA$ via the natural action of $\alpha(k')$ on
$i_x^*\sA$.\footnote{Let $m(w)$ be the maximal ideal of $\bC[Z]$ corresponding
to $w$ and let $m(x)$ be the maximal ideal of $\bC[\bcO]$ corresponding to $x$.
Then the map  $\phi : w \mapsto x$ gives a $\bC[\bcO]$-algebra isomorphism: $L :
  \xymatrix{\bC[\bcO]/m(x) \ar[r]^{\simeq \ \ }& \bC[Z]/m(w) = \bC}$. The
  group $S_w$ acts on the right hand side while the group $K_{x}$ acts
  on the left hand side.  These two actions are compatible in the
  sense that for $(\alpha(k'),k') \in S_w \subset K_x \times
  K_{x'}'$, we have $L \circ \alpha(k') = (\alpha(k'),k') \circ L$.

  Similarly $K_\bC'$ acts on $\bC[Z]\otimes_{\bC[\bcO]} A$ via
  translation on $\bC[Z]$ while $K_\bC$ acts via the tensor product of
  its action on $A$ and the translation action on $\bC[Z]$. Then $\chi
  = (\bC[Z]/m(w)) \otimes_{\bC[\bcO]} A \simeq (\bC[\bcO]/m(x))
  \otimes_{\bC[\bcO]} A = i_{x}^*\sA $.  Let $(\alpha(k'),k') \in
  S_w$.  Then it acts on the right hand side via its natural action of
  $\alpha(k')$ on $i_x^*\sA$.  }

Putting the above into \eqref{eq15} gives
\begin{equation} \label{eq16}
i_{x'}^*\sB =(\sZ(Z_{x'}))^{K_\bC}  =
(\Ind_{K_x\times_\alpha K'_{x'}}^{K_\bC \times K'_{x'}} \chi)^{K_\bC}
\simeq \chi\circ \alpha
\end{equation}
as $K'_{x'}$-modules. Indeed if $f \in (\Ind_{K_x\times_\alpha
  K'_{x'}}^{K_\bC \times K'_{x'}} \chi)^{K_\bC}$, then $f : K_\bC
\times K'_{x'} \rightarrow V_\chi$ satisfies $f(k,k') =
\chi(\alpha(k'))f(\alpha(k')^{-1} k,1) = \chi(\alpha(k'))
f(1,1)$. Hence $f$ is uniquely determined by $f(1,1) \in V_\chi$. This
proves the isomorphism on the right in \eqref{eq16}.  It also
completes the proof of the lemma.
\end{proof}

\subsection{} Let $\rho' = \Theta(\rho)$, $\bfA = \varsigma|_{\wtK}
\otimes \Gr V_{\rho^*}$ and $\bfB = \varsigma|_{\wtK'} ^{-1} \otimes
\Gr\Theta(V_\rho)$ as before.  For a subset~$Z$ of~$\frakp^*$, we let
$I(Z)$ denote the ideal of $\cS(\frakp)$ vanishing on~$Z$.

\begin{prop} \label{P8} There is a finite filtration $0 = \bfA_0
  \subset \cdots \subset \bfA_l\subset \bfA_{l+1}\subset \cdots
  \subset \bfA_n = \bfA$ of $(\cS(\fpp),K_\bC)$-modules with the
  following property: For each $l$, there is a $K_\bC$-orbit $\cO_l$
  such that the annihilator ideal of $\bfA_{l}/\bfA_{l-1}$ in
  $\cS(\fpp)$ is the ideal $I(\overline{\cO_l})$.

  In particular $\bfA_{l}/\bfA_{l-1}$ is a
  $\bC[\overline{\cO_l}]$-module and $\bigcup_{l=1}^n \overline{\cO_l}
  = \AV(\rho^*)$.
\end{prop}

\subsection*{Remark} We warn that the orbit $\cO_l$ may not be
connected since $K_\bC$ may not be connected. Furthermore $\cO_l$ may
not be an open orbit in $\AV(\rho^*)$.

\begin{proof}
The proof essentially follows that of Lemma 2.11 in \cite{Vo89}.

Let $K_0$ be the connected component of $K_\bC$. The set of associated
primes of $\bfA$ is finite.  The connected group $K_0$ acts trivially
on this finite set of associated primes.

Let $a \in \bfA$ such that its annihilator ideal $\cqq = \Ann_{\cS(\fpp)}(a)$ is a minimal associated prime of
$\bfA$.  Let $\bfA_1 = \cS(\fpp ) K_\bC a$ be the $(\cS(\fpp),K_\bC)$-submodule in $\bfA$ generated by $a$.  Let
$\rV(\cqq)$ be the subset of $\frakp^*$ cut out by $\cqq$. Since
$\rV(\cqq)$ is irreducible and $K_0$-invariant, it is the closure of
single $K_0$-orbit $\cO_0$.  Let $\cO_1 = K_\bC \cO_0$.

We claim that $\Ann_{\cS(\fpp )}(\bfA_1) = I(\overline{\cO_1})$.
Indeed,
\begin{eqnarray*}
\Ann_{\cS(\fpp )}(\bfA_1) & = & 
  \bigcap_{k \in K_\bC} \Ann_{\cS(\fpp )} (k \cdot a) =
  \bigcap_{k \in K_\bC} k \cdot \cqq = \bigcap_{[k] \in K_\bC/K_0}
  [k]\cdot \cqq \\ & = & I(\bigcup_{[k]\in K_\bC/K_0}[k]\cdot
  \rV(\cqq)) = I(\overline{\cO_1}).
\end{eqnarray*}
The above second last equality holds because $\bigcap_{[k] \in
  K_\bC/K_0} [k]\cdot \cqq$ is a finite intersection of prime
ideals. The last equality holds by the definition of $\cO_1$. This
proves our claim.

\smallskip

Now, we could construct $\bfA_l$ and $\cO_l$ inductively by applying
the above construction to the $(\cS(\fpp ), K_\bC)$-module
$\bfA/\bfA_{l-1}$.  This procedure will eventually stop because $\bfA$
is a finitely generated module over the Noetherian ring $\cS(\fpp )$.
\end{proof}

Let $\bfA_l$ be as in Proposition \ref{P8} and let $\bfA^l = \bfA_l/\bfA_{l-1}$. It is a finitely generated
$\bC[\overline{\cO_l}]$-module and we let $\sA^l$ be its associated
coherent sheaf on $\overline{\cO_l}$.

By Lemma \ref{L3}, $\bfB = (\bC[W]\otimes_{\cS(\fpp )}
  \bfA )^{K_\bC}$.  Let $\bfB_l = (\bC[W]\otimes_{\cS(\fpp )}
  \bfA_l)^{K_\bC}$. Since $\phi$ is flat by Theorem~\ref{thm:Kos} and
  taking $K_\bC$-invariants is exact, we may identify $\bfB_l$ with a
  submodule of~$\bfB$. Hence $\bfB_l$ is an
  $(\cS(\fpp'),K_\bC')$-equivariant filtration of $\bfB$.  We set
  $\bfB^l = \bfB_{l}/\bfB_{l-1}$. Then
\[
\bfB^l = \bfB_{l}/\bfB_{l-1} = (\bC[W]\otimes_{\cS(\fpp)} (\bfA_{l} /
\bfA_{l-1}))^{K_\bC} = (\bC[W]\otimes_{\cS(\fpp)} \bfA^l)^{K_\bC}.
\]

\subsection{}
We define a partial ordering on the $K_\bC$-orbits by containments in
the Zariski closures.  Let $\{ \cO_{l_1}, \ldots, \cO_{l_r} \}$ be the
set of (distinct) maximal nilpotent $K_\bC$-orbits appearing in
Proposition \ref{P8}.  For each $\cO_{l_j}$ in this set, we fix a
closed point $x_j\in \cO_{l_j}$ and define  the 
$K_{x_j}$-module
\begin{equation}\label{eq:chigrA}
\chi(x_j,\Gr\sA) = \bigoplus_{\cO_{l_j} = \cO_l} i_{x_j}^*\sA^l.
\end{equation}

Let $m_j= \dim_\bC \chi(x_j,\Gr\sA)$. The integer $m_j$ is independent
of the choice of $x_j \in \cO_{l_j}$. Moreover $m_j\neq 0$.  Indeed
all $K_{x_j}$-modules on the right hand side of \eqref{eq:chigrA} are
non-zero because $\Supp(\sA^l) = \rV(\Ann_{\cS(\fpp)} \bfA^l) =
\overline{\cO_{l}} = \overline{\cO_{l_j}}$.

Recall that $\bfA = \varsigma|_{\wtK}\otimes \Gr V_{\rho^*}$. Let 
\[
\chi_{x_j} = \varsigma|_{\wtK}^{-1} \otimes \chi(x_j,\Gr\sA).
\]
Then $\set{(\cO_{l_j},x_j, \chi_{x_j})}$ is the set of orbit data
attached to the filtrations given by Proposition~\ref{P8}.

Now the
associated cycle of $\rho^*$ is 
\[
\AC(\rho^*) = \AC(\bfA) = \sum_{j=1}^r m_j [\overline{\cO_{l_j}}].
\]
and the associated variety is $\AV(\rho^*) = \bigcup_{j=1}^r
\overline{\cO_{l_j}}$.

\subsection{}{\it Proof of Theorems \ref{TC} and
  \ref{TD}.} \label{S44b}
First we observe following lemma. 

\begin{lemma}\label{lem:cpxorb}
  Let $\set{\cO_{l_j} : j = 1, \ldots, r}$ be the set of all distinct
  (open) maximal $K_\bC$-orbits in~$\AV(\rho^*)$. Then
  $\set{\theta(\cO_{l_j}) : j = 1, \ldots, r}$ forms the set of all
  distinct (open) maximal $K'_\bC$-orbits in $\theta(\AV(\rho^*))$.
\end{lemma}
\begin{proof} By Theorem~\ref{thm:TLorb}(i) the map $\theta\colon
  \frakN_{K_\bC}(\fpp^*)\to \frakN_{K'_\bC}(\fpp'^*)$ is injective so
  all the $\theta(\cO_{l_j})$'s are distinct. We also have
  $\theta(\AV(\rho^*)) = \phi'(\phi^{-1}(\bigcup_{j=1}^r
  \overline{\cO_{l_j}})) = \bigcup_{j=1}^r
  \theta(\overline{\cO_{l_j}}) = \bigcup_{j=1}^r
  \overline{\theta(\cO_{l_j})}$. It suffices to show that $\dim
  \theta(\cO_{l_j}) = \dim \theta(\AV(\rho^*))$.

  By Theorem~8.4 in \cite{Vo89}, every $K_\bC$-orbit $\cO_{l_j}$
  generates the same $G_\bC$-orbit $\cO_\bC$ in~$\fgg^*$.  Indeed
  $\overline{\cO_\bC}$ is the variety cut out by
  $\Gr(\Ann_{\cU(\fgg)}(\rho^*))$.

  Nilpotent $K_\bC$-orbits for classical groups are parametrized by
  signed Young diagrams.  In particular the underlying Young diagrams
  of different $\cO_{l_j}$'s are the same and they have the same
  dimension equals to $\frac{1}{2} \dim_\bC \cO_\bC$.  By
  \cite{Ohta:1991C}, the signed Young diagram of the
  orbit~$\theta(\cO_l)$ is obtained by adding a column to the signed
  Young diagram of $\cO_{l_j}$. Hence every $K_\bC'$-orbit
  $\theta(\cO_{l_j})$ generates the same $G_\bC'$-orbit $\cO'_\bC$
  in~$\fgg'^*$ and its dimension is $\frac{1}{2} \dim_\bC \cO'_\bC$.
  This proves that $\dim \theta(\cO_{l_j}) = \dim\theta(\AV(\rho^*))$
  and completes the proof of the lemma. In fact $\overline{\cO_\bC'} =
  \theta_\bC(\overline{\cO_\bC})$ where $\theta_\bC$ was defined
  after~\eqref{eq2a} (see \cite{DKP}, \cite{Daszkiewicz:2005} or
  \cite{KP}).
\end{proof}

Let $\cO'_{l} = \theta(\cO_l)$.  By Theorem \ref{TB}, $\AV(\rho')
\subseteq \bigcup_{j = 1}^r \theta(\overline{\cO_{l_j}}) =
\bigcup_{j=1}^r \overline{\cO'_{l_j}}$.

Let $\sB_l$ and $\sB^l$ be the associated coherent sheaves of $\bfB_l$
and $\bfB^l$ respectively. Now we apply Lemma~\ref{lem:isotropy1} to
$\bfB^l$ and we have
\[
\chi(x'_j, \Gr\sB) := \bigoplus_{\cO'_l = \cO'_{l_j}} i_{x'_j}^*\sB^l
\simeq \bigoplus_{\cO_l = \cO_{l_j}} \alpha_j^*(i_{x_j}^*\sA^l)
\]
where $x'_j$ and $\alpha_j$ are $x'$ and $\alpha$ respectively in
Lemma~\ref{lem:isotropy1}.

Since $\bfA = \varsigma|_{\wtK} \otimes \Gr V_{\rho^*}$ and $\bfB =
\varsigma|_{\wtK'} ^{-1} \otimes \Gr\Theta(V_\rho)$, the isotropy
representation of $\Theta(\rho)$ at $x'_j$ with respect to the
filtration $\sB_l$ is
\[
\chi_{x'_j} = \varsigma|_{\wtK'} \otimes \chi(x'_j, \Gr\sB) = 
\varsigma|_{\wtK'}\otimes \chi(x_j,\Gr \sA)\circ \alpha_j = 
\varsigma|_{\wtK'}\otimes  (\varsigma|_{\wtK}\otimes \chi_{x_j})\circ \alpha_j.
\] 
In particular, $\chi_{x'_j} \neq 0$ since $\chi_{x_j}\neq 0$.
Therefore $\{(\cO'_{l_j}, x'_j, \chi_{x'_j}) : j = 1, \ldots, r \}$
forms the set of orbit data attached to the filtration $\sB_l$.  This
proves Theorem~\ref{TC}.

Now 
\[
\AC(\Theta(\rho)) = \sum_{j = 1}^r (\dim \chi_{x'_j})
[\overline{\cO'_{l_j}}] = \sum_{j = 1}^r m_j
[\overline{\theta(\cO_{l_j})}] = \theta(\AC(\rho^*)).
\]
This proves Theorem \ref{TD}. \qed

\medskip

The proof also shows that the theta lift of a Harish-Chandra module in
stable range is nonzero since $\chi_{x_j'} \neq 0$.

\subsection{} \label{S44} {\it Proof of Corollary \ref{CE}.}  We
recall $\rho' = \Theta(\rho)$.  From the proof of
Lemma~\ref{lem:cpxorb},
\[
\overline{G_\bC' \AV(\rho')} = \overline{\cO_\bC'} =
\theta_\bC(\overline{\cO_\bC}) = \theta_\bC(\overline{G_\bC
  \AV(\rho^*)}) = \theta_\bC(\VC(\rho^*)) = \theta_\bC(\VC(\rho)).
\]
The last equality follows from Proposition \ref{P24}.  Although
$\rho'$ may not be irreducible, we claim that $\VC(\rho') =
\overline{G_\bC' \AV(\rho')}$ and this would prove the
corollary. First $\VC$ is an additive map, i.e. $\VC(B) = \VC(A) \cup
\VC(C)$ for every exact sequence $0\to A\to B\to C\to 0$.  This is
well known to the experts (for example see Lemma 1.5 in \cite{Be}),
which follows by taking the graded version of
\[
\Ann_{\cU(\fgg')}(A)\Ann_{\cU(\fgg')}(C) \subseteq
\Ann_{\cU(\fgg')}(B) \subseteq \Ann_{\cU(\fgg')}(A)\cap
\Ann_{\cU(\fgg')}(C).
\]
Next let $\rho_1', \ldots, \rho_s'$ be all the
irreducible subquotients of the $(\fgg',\wtK')$-module $\rho'$ of
finite length.  Using Theorem~8.4 in \cite{Vo89} again,
\[
\VC(\Ann \rho') = \bigcup_{k = 1}^s \VC(\Ann \rho_k') =
\bigcup_{k = 1}^s \overline{G_\bC' \AV(\rho_k')} = \overline{G_\bC'
  \bigcup_{k = 1}^s \AV(\rho_k')} = \overline{G_\bC' \AV(\rho')}.
\]
This proves our claim and Corollary \ref{CE}.
\qed

\section{The $K$-spectrum equation} \label{S5}

In this section, we suppose $(G,G')$ is in the stable range with $G$
the smaller member excluding~\eqref{eq:ddagger}. We will also retain
the notation in the previous section. The objective of this section is
to prove Proposition \ref{P17} which implies Theorem \ref{TF}.

\medskip

\subsection{}
Let $x \in \cO$ and let $\chi_x$ be a finite dimensional rational
representation of $K_x$ as in Section~\ref{S13}. We recall Theorem~2.7
in \cite{CPS} that there is an equivalence of categories between the
category of rational representations of $K_x$ and the category of
certain $K_\bC$-equivariant sheaves on $\cO \simeq K/K_x$. Let $\sL$
be the $K_\bC$-equivariant sheaf on $\cO$ corresponding to
$\chi_x$. We assume that $\sL$ is generated by its global sections
(c.f.~\eqref{eq:condchi}). Let $i_\cO : \cO \rightarrow \fpp^*$ denote
the inclusion map and let $\sA = (i_\cO)_* \sL$. We also set
\begin{equation} \label{eq21}
A := \sA (\fpp^*) = \sL(\cO) = \Ind_{K_x}^{K_\bC} \chi_x
\end{equation}
as an $(\cS(\fpp), K_\bC)$-module. Clearly $A$ is a
$(\bC[\bcO],K_\bC)$-module.

Let $\cO' = \theta(\cO)$. We fix a $w \in W$ such that $x = \phi(w)
\in \cO$ and $x' = \phi'(w) \in \cO'$ in \eqref{eq2}. Let $\alpha
\colon K_{x'}' \to K_{x}$ be the map defined in Proposition
\ref{prop:alpha} in Appendix \ref{SB}. Let $\chi_{x'} = \chi_x\circ
\alpha$ be the representation of $K_{x'}'$. Let~$\sL'$ be the
$K_\bC'$-equivariant sheaf on $\cO'$ corresponding to~$\chi_{x'}'$. We
define the $(\cS(\fpp'),K'_\bC)$-module
\[
B = (\bC[W] \otimes_{\cS(\fpp )} A)^{K_\bC}.
\]
By Lemma \ref{lem:isotropy1}, $B$ is a $(\bC[\bcOp],K'_\bC)$-module.

\begin{prop} \label{P17} Suppose $(G,G')$ is in the stable range with
  $G$ the smaller member excluding~\eqref{eq:ddagger}.  Let $i_{\cO'}
  : \cO' \rightarrow \bcOp$ denote the inclusion map.  Then the sheaf
  $(i_{\cO'})_* \sL'$ is the coherent sheaf associated to the
  $(\bC[\bcOp],K_\bC')$-module $B$.
\end{prop}

\begin{proof}
  Let $Y := W\times_{\fpp } \cO$. By Lemma \ref{lem:rednorm}, $Y$ is a
  reduced scheme.  We consider following diagram where $Z^\circ =
  (\phi')^{-1}(\cO') \cap Y$.
\[
\xymatrix{
& Z^\circ \ar@{^(->}[d] \ar[r]^{\phi'} & \cO' \ar@{^(->}[d]^{i_{\cO'}}  \\
 \cO \ar@{^(->}[d]^{i_{\cO}}& Y \ar[l]_{\phi|_Y} \ar@{^(->}[d]^{i_Y}  \ar[r]^{\phi'|_Y} &  \bcOp \ar@{^(->}[d] \\
 \fpp^* & \ar[l]_{\phi} W\ar[r]^{\phi'} &\fpp'^*&
}
\]
Since $\phi$ is flat, Proposition 9.3 in Chapter 3 in \cite{HS} gives
\[
\phi^*(i_{\cO})_* \sL = (i_Y)_* (\phi|_Y)^* \sL
\]
as sheaves on $W$. Let $\sQ = (\phi|_Y)^* \sL$ and let $\sB$ be the
quasi-coherent sheaf on $\bcOp$ associated to $B$.  Note that $
((\phi'|_Y)_* \sQ(\bcOp))^{K_\bC} = (\phi'_* {i_Y}_*
\sQ(\fpp'^*))^{K_\bC} = B = \sB(\bcOp)$.  By the exactness of taking
$K_\bC$-invariants and the fact that $\bC[\bcOp]$ is
$K_\bC$-invariant, we have $\sB = ((\phi'|_Y)_*\sQ)^{K_\bC}$,
i.e. $\sB(U) = \left(\sQ((\phi'|_Y)^{-1}(U))\right)^{K_\bC}$ for every open set
$U\subset \bcOp$. 

\begin{lemma}
We have $\sQ(Y) = \sQ(Z^\circ)$.
\end{lemma}

\begin{proof}
The proof is similar to Theorem 4.4 in \cite{CPS}.
  
Since $\sL$ is locally free on $\cO$, $\sQ$ is locally free on
$Y$. Hence
\[
\depth \sQ_y = \depth \sO_{Y,y}\] for any $y\in Y$.  Let $\partial
Z^\circ = Y- Z^\circ$.  Let $H_{\partial Z^\circ}^i(\sQ)$
(resp. $\sH_{\partial Z^\circ}^i(\sQ)$) be the cohomology group
(resp. cohomology sheaf) of $Y$ with coefficient in $\sQ$ and
support in $\partial Z^\circ$.  By Lemma~\ref{lem:codim2} in
Appendix \ref{SB}, $\codim(Y, \partial Z^\circ) \geq 2$.  By Lemma
\ref{lem:rednorm}(ii), $Y$ is a normal scheme so it satisfies
 Serre's (S2) condition. Therefore
\[
\depth_y \sQ_y = \depth \sO_{Y,y} \geq \min\set{\dim\sO_{Y,y},2} = 2
\]  for all $y \in \partial  Z^\circ$.
By a vanishing theorem of Grothendieck (see \cite{GH}*{Theorem 3.8}),  
\[
\sH_{\partial  Z^\circ}^i(\sQ) = 0 \mbox{ \ for \  } i = 0, 1.
\] 
By Proposition~1.11 in \cite{GH}, a spectral sequence argument
implies $H^0(Y, \sQ) \simeq H^0(Z^\circ, \sQ)$ as required.
\end{proof}

We continue with the proof of Proposition \ref{P17}. By the above lemma
\[
B = \sB (\bcOp) = (\sQ(Y))^{K_\bC} = (\sQ(Z^\circ))^{K_\bC} =
\sB (\cO').
\]
By \eqref{eq16} we have
\[
\sB (\cO') = \Ind_{K'_{x'}}^{K_\bC'} (\chi_x \circ \alpha) =
(i_{\cO'})_* \sL'(\bcOp).
\]
Hence $\sB (\bcOp) = (i_{\cO'})_* \sL'(\bcOp)$. Since both $\sB$ and
$(i_{\cO'})_* \sL'$ are quasi-coherent sheaves over the affine scheme
$\bcOp$, $\sB = (i_{\cO'})_* \sL'$ and this completes the proof of
Proposition \ref{P17}.
\end{proof}

\subsection{} {\it Proof of Theorem \ref{TF}.} \label{S53} By
\eqref{eq21}, $A|_K = \varsigma|_{\wtK} \otimes V_{\rho^*}|_{\wtK}$ as
$K$-modules. Therefore
\[
\begin{array}{rcll}
  \varsigma|_{\wtK'}^{-1} \otimes \rho'|_{\wtK'} & = & 
  \varsigma|_{\wtK'}^{-1} \otimes \Gr(\rho')|_{\wtK'} = (
  \varsigma|_{\wtK} \otimes V_{\rho^*}|_{\wtK}
  \otimes \cH)^{K} &  \mbox{(by Proposition \ref{P4})} \\ 
  & = & (A|_K \otimes \cH)^{K} = B|_{K'}  &  \mbox{(by Proposition
    \ref{P4} again)} \\
  & = & (i_{\cO'})_* \sL'(\bcOp) & \mbox{(by Proposition \ref{P17})} \\
  & = & \sL'(\cO') =  \Ind_{K_{x'}'}^{K_\bC'} \left( (\varsigma|_{\wtK} \otimes
    \chi_x)\circ \alpha \right).
\end{array}
\]
Twisting the above equation by $\varsigma|_{\wtK'}$ proves Theorem
\ref{TF}. \qed

\section{Admissible data} \label{sec:adm} In this section we will show
that the theta lift of an admissible data is still an admissible
data. We continue to assume that $(G,G')$ is an irreducible type I
dual pair in the stable range where $G$ is the smaller member.

Let $\cO$ be a nilpotent $K_\bC$-orbit in $\fpp^*$ as in
\eqref{eq2}. Let $\cO' = \theta(\cO)$.  Let $w \in W$ such that $x =
\phi(w)\in \cO$ and $x' = \phi'(w)\in \cO'$. Let $\alpha\colon
K'_{x'}\to K_{x}$ be the map defined by
Proposition~\ref{prop:alpha}. Proposition \ref{PG} follows from the
next proposition.

\begin{prop} \label{prop:adm}
  Suppose $(G,G')$ is in the stable range where $G$ is the smaller
  member. Let~$\chi_{x}$ be an admissible representation of $\wtK_x$
  as defined in Section \ref{S16}. We set
\[
  \chi_{x'} : =  \varsigma|_{\wtK'_{x'}} \otimes (\varsigma|_{\wtK_x}\otimes
  \chi_x)\circ \alpha.
\] 
Then $\chi_{x'}$ is an admissible representation of $\wtK'_{x'}$.
\end{prop}

\noindent{\it Proof.}
We have to verify that
\[
\chi_{x'}(\exp(X')) = \det(\Ad^*(\exp(X'/2))|_{(\fkk'/\fkk'_{x'})^*})
\qquad \forall X'\in \fkk'_{x'}.
\]
Since $\chi_{x}$ is admissible, it reduces to the following lemma after
taking square of above equation.

\begin{lemma}
As $\fkk'_{x'}$-modules,
\[
\topform (\fkk'/\fkk'_{x'}) \simeq (\topform (\fkk/\fkk_x) \circ
\alpha) \otimes \varsigma|_{K'}^{-2} \otimes (\varsigma|_K^{-2} \circ
\alpha).
\]
\end{lemma}

\begin{proof}
  Let $E = K_\bC K'_\bC w$, $F = \phi^{-1}(x)$, and $F' =
  \phi^{-1}(x')$. Let $S_w := \Stab_{K_\bC\times K'_\bC}(w) =
  \Set{(\alpha(k'),k')|k'\in K'_{x'}} \simeq K'_{x'}$.  Let $\rT_w F'$
  denote the tangent space of $F'$ at $w$ etc. We have following two
  exact sequences of $S_w$-modules:
\[
\xymatrix{
0\ar[r]& \rT_w F' \ar[r] & \rT_w E \ar[r]& \rT_{x'} \cO'\ar[r] &0 
}
\]
and 
\[
\xymatrix{
0\ar[r]& \rT_w F \ar[r] & \rT_w E \ar[r]& \rT_{x} \cO \ar[r] & 0 
}.
\]
Here $S_w$ acts on $\rT_{x} \cO$ (resp. $\rT_{x'} \cO'$) via the
projection $S_w \rightarrow K_{x}$ (resp. $S_w
\xrightarrow{\sim}K_{x'}'$). Since $S_w \simeq K_{x'}'$, the
above are also exact sequences of $K_{x'}'$-modules.

By Proposition \ref{prop:alpha}(i) $\rT_w F' \simeq \fkk$. The
$\fkk'_{x'}$-action on $\topform \fkk$ is trivial since $\fkk$ is
reductive.  Therefore
\begin{equation}\label{eq:tan1}
  \topform \rT_{x'}\cO' \simeq \topform \rT_{w} E 
  \simeq \topform \rT_x\cO\otimes \topform \rT_w F
\end{equation}
as $\fkk_{x'}'$-modules.  Since we are in the stable range, $\phi : W
\to \fpp^*$ is a submersion at every point $w\in W$.  We have
following exact sequence of $K'_{x'}$-modules:
\[
\xymatrix{
0\ar[r]& \rT_w F \ar[r] & \rT_w W \ar[r]& \rT_{x} \fpp^* \ar[r] & 0. 
}
\]
Since $\rT_x\fpp^* \simeq \fpp^*$ and $\fkk'_{x'}$ acts trivially on
$\topform \fpp^*$, we have
\begin{equation}\label{eq:tan2}
  \topform \rT_w W \simeq \topform \rT_w F. 
\end{equation}
Combining \eqref{eq:tan1}, \eqref{eq:tan2}, $\rT_w W \simeq W$,
$\rT_x\cO \simeq \fkk/\fkk_x$ and $\rT_{x'}\cO' \simeq
\fkk'/\fkk'_{x'}$, we have
\begin{equation} \label{eq28}
\topform (\fkk'/\fkk'_{x'}) = (\topform  \fkk/\fkk_x \circ \alpha) 
\otimes \topform W.
\end{equation}

We view $u \in \rU_\bC$ as a linear transformation on $W$. By our
choice of oscillator representation,
\[
\varsigma^{-2}(u) = \det(u|_W).
\]
Hence the action of $k' \in K'_{x'}$ on $\topform W$ is
\[
\det((k',\alpha(k'))|_W) =
\varsigma|_{K'}^{-2}(k')\otimes (\varsigma|_{K}^{-2} \circ \alpha (k')).
\]
Putting this into \eqref{eq28} proves the lemma and Proposition
\ref{prop:adm}.
\end{proof}

\appendix

\section{The Geometry of theta lifts of nilpotent orbits}\label{SB}

\subsection{The Fock model} \label{sec:Fock}
We retain the notation in Section~\ref{S11} where $(W_\bR, \langle ,
\rangle)$ is a symplectic space and we have fixed a maximal compact
subgroup $\rU \subset \Sp(W_{\bR})$.  It is well known that there are
two oscillator representations.  We will specify our choice of
oscillator representation in this paper by describing its Fock model.

First we fix a square root of $-1$, say $i$.  The centralizer of $\rU$
in $\Sp(W_\bR)$ is isomorphic to $\rU(1)$.  For an element $J$ in the
centralizer such that $J^2 = -1$, we define a complex structure on
$W_\bR$ such that $i\cdot v := Jv$ for all $v\in W_\bR$. We denote the
corresponding complex vector space by $W$.  Then $\inn{v_1}{v_2}_H :=
\inn{Jv_1}{v_2} + i \inn{v_1}{v_2}$ defines a Hermitian form
on~$W$. There are two choices of $J$ and we choose the one such that
$\inn{}{}_H$ is positive definite. Now $\rU$ is the unitary group
$\rU(W,\inn{}{}_H)$. Its complexification is $\rU_\bC = \GL(W)$ and
the covering group of $\rU_\bC$ is $\wtrU_\bC = \Set{(g,z)\in \GL(W)
  \times \bC^{\times}|\det(g) = z^2}$. We identify $\wtrU$ with the
inverse image of $\rU$ in $\wtrU_\bC$.  We fix the oscillator
representation $\omega$ such that its Fock module or
$(\fgg,\wtrU)$-module $\sY$ is isomorphic to $\bC[W]$ such that
$(\omega(\tg) f)(v) = z^{-1} f(g^{-1} v) $ for all $\tg = (g,z) \in
\wtrU_\bC$ and $f\in \bC[W]$. In particular the minimal $\wtrU$-type
is one dimensional consisting of constant functions on $W$ and $\wtrU$
acts on it via the character $\varsigma(\tg) = z^{-1}$ where $\tg =
(g,z)\in \wtrU$.

\subsection{The moment maps} \label{SA1} Let $(G,G')$ denote a type I
irreducible reductive dual pair in $\Sp(W_\bR)$ as in Table
\ref{tab:sstabncp}. Let $\sp = \Lie(\Sp(W_\bR))_\bC$ and $\sp^{(1,1)}
= \Lie(\rU)_\bC$. Under the adjoint action of $\rU$, $\sp =
\sp^{(2,0)} \oplus \sp^{(1,1)} \oplus \sp^{(0,2)}$. Here $\sp^{(2,0)}$
is an abelian Lie subalgebra acting on $\bC[W]$ via multiplication by
degree two polynomials. In particular, we have $\cS(\sp^{(2,0)})
\rightarrow \bC[W]$ by $p \mapsto p\cdot 1$. Let $\fgg = \fkk \oplus
\fpp$ be the complexified Cartan decomposition of $G$. The composition
$\fpp \hookrightarrow \sp \twoheadrightarrow \sp^{(2,0)}$ induces an
algebra homomorphism $\cS(\fpp) \to \cS(\sp^{(2,0)})$. Composing the
two maps gives $\phi^* : \cS(\fpp) \to \bC[W]$ which defines $\phi
\colon W \to \fpp^* = \Spec(\cS(\fpp))$.  Similarly we define $\phi'
\colon W \to \fpp'^* = \Spec(\cS(\fpp'))$. Hence we have
\[
\xymatrix{
\fpp^* & \ar[l]_{\phi} W \ar[r]^{\phi'} & \fpp'^*.}
\]
The maps $\phi$ and $\phi'$ are called the {\it moment maps}.

\smallskip

We describe explicitly the moment maps in Table~\ref{tab:noncpt1}
below. Here $J_{2p}$ is the skew symmetric $2p$ by $2p$ matrix
$\left(\begin{smallmatrix} 0 & 1 \\ -1 & 0 \end{smallmatrix}\right)$.

\begin{table}[htbp]
  \centering \footnotesize
  \begin{tabular}{cc|c|cc}
    $G$ & $G'$ & ${\Wcp}$ &$\fpp^*$ & $\fpp'^*$ \\
    & & $w\in W$& $\phi(w)$ & $\phi'(w)$\\
\hline
\multirow{2}{*}{$\Sp(2n,\bR)$} & \multirow{2}{*}{$\rO(p,q)$} &  
$M_{p,n}\times M_{q,n}$ & $\Sym_n\times \Sym_n$ & $M_{p,q}$\\
& & $(A,B)$ & $(A^TA, B^TB)$ & $AB^T$ \\
\hline
\multirow{2}{*}{ $\rU(n_1,n_2)$} & \multirow{2}{*}{$\rU(p,q)$} & $M_{p,n_1}\times M_{p,n_2}\times
M_{q,n_1}\times M_{q,n_2}$&
 $M_{n_1,n_2}\times M_{n_2,n_1}$& $M_{p,q} \times
M_{q,p}$ \\
& & $(A,B,C,D)$ & $(A^TB, D^TC)$ & $(AC^T,DB^T)$ \\
\hline
\multirow{2}{*}{$\rO^*(2n)$} & \multirow{2}{*}{$\Sp(p,q)$} & $M_{2p,n}\times M_{2q,n}$ &
$\Alt_n\times\Alt_n$ & $M_{2p,2q}$\\
& & $(A,B)$ & $(A^T J_{2p} A,B^TJ_{2q} B)$ & $AB^T$ \\
\hline
\multirow{2}{*}{$\Sp(2n, \bC)$} & \multirow{2}{*}{$\rO(p,\bC)$} & 
$M_{p,2n}$ &
$\Sym_{2n}$ & $\Alt_p$\\
& & $A$ & $A^T A$ & $A J_{2n} A^T$
\end{tabular}
\caption{Moment maps for non-compact dual pairs.}
\label{tab:noncpt1}
\end{table}

The following fact is true for every reductive dual
pairs, not necessarily in the stable range.  The moment map factors
through the affine quotient:
\begin{equation} \label{eq22} \xymatrix@C=4em{ W\ar@{->>}[r]
    \ar@/^2pc/[rr]^{\phi'}& W/K_\bC \ar@{^{(}->}[r]^{i_{W/K_\bC}} &
    \fpp'^*.  }
\end{equation}
By the First Fundamental Theorem of classical invariant theory,
$\bC[W]^{K_\bC}$ is a quotient of $\cS(\fpp')$, i.e. $i_{W/K_\bC}$ is
a closed embedding. For every $K_\bC$-invariant closed subset $E
\subseteq W$, its image in $W/K_\bC$ is closed by Corollary 4.6 in
\cite{PV}. This implies that $\phi'(E)$ is closed in $\fpp'^*$.  Hence
for every $K_\bC$-invariant closed subset $S \subseteq \fpp^*$,
$\theta(S) := \phi'(\phi^{-1}(S))$ is a $K'_\bC$-invariant closed
subset of $\frakp'^*$.

\smallskip

\subsection{} 
We recall the nilpotent cone $\nullp = \Set{x \in \fpp^*|0\in
  \overline{K_\bC \cdot x}}$.  Let $\frakN_{K_\bC}(\fpp^*)$ be the set
of nilpotent $K_\bC$-orbits in $\fpp^*$.  We define $\nullpp$ and
$\frakN_{K'_\bC}(\fpp'^*)$ in the same way.

We summarize some results in \cite{Ohta:1991C},
\cite{Daszkiewicz:2005} and \cite{NOZ1}.

\begin{thm} \label{thm:TLorb} Let $(G,G')$ be a reductive dual pair in
  stable range where $G$ is the smaller member as in
  Table~\ref{tab:sstabncp}.
\begin{enumerate}[(i)]
\item For any nilpotent $K_\bC$-orbit $\cO$ in $\fpp^*$, there is a
  nilpotent $K'_\bC$-orbit $\cO'$ in $\fpp'^*$ such that
  \[
  \phi'(\phi^{-1}(\bcO)) = \overline{\cO'}.
  \]
  This defines an injective map $\theta\colon \frakN_{K_\bC}(\fpp^*)
  \to \frakN_{K'_\bC}(\fpp'^*)$ given by $\cO \mapsto \cO'$. This map
  is called the theta lifting of nilpotent orbits.

\item Theta lifting of nilpotent orbits preserves closure relation,
  i.e. if $\cO_0 \subset \overline{\cO}$ then $\theta(\cO_0)\subset
  \overline{\theta(\cO)}$. \qed
\end{enumerate}
\end{thm}

We refer to Table \ref{tab:noncpt1} where $W$ is written as a product
of matrix spaces.  Let $W^\circ $ be the open dense subset of elements
in $W$ whose every component has full rank. Before we discuss the
finer structures of orbits, we state following lemma.

\begin{lemma}\label{lem:fullrank}
Let $(G,G')$ be a reductive dual pair in the stable range as in
Table \ref{tab:sstabncp}.
\begin{enumerate}[(i)]
\item We have $\phi'^{-1}(\phi'(W^\circ)) = W^\circ$.

\item For any $x' \in \phi'(W^\circ)$, $\phi'^{-1}(x') =
  \phi'^{-1}(x')\cap W^\circ$ is a single $K_\bC$-orbit where $K_\bC$
  acts freely.

\item  For any $x \in \phi(W^\circ)$, $\phi^{-1}(x) \cap W^\circ$ is a
  single $K_\bC'$-orbit.

\item We have one-to-one correspondences of the following sets of orbits
\[
\begin{array}{rcccl}
  \Set{K_\bC\text{-orbits in } \phi(W^\circ)} & \leftrightarrow & 
  \Set{K_\bC\times K'_\bC\text{-orbits in } W^\circ} &
  \leftrightarrow & \Set{K'_\bC\text{-orbits in } \phi'(W^\circ)} \\
  \phi(C) & \mapsfrom & C & \mapsto &  \phi'(C) \\
  \cO & \mapsto & \phi^{-1}(\cO) \cap W^\circ & & \\
  & & \phi'^{-1}(\cO') = \phi'^{-1}(\cO') \cap W^\circ & \mapsfrom & \cO'.
\end{array}
\]
\end{enumerate}
\end{lemma}

\begin{proof}
  The proof for each dual pair is similar so we will give the proof
  for the first pair in Table~\ref{tab:sstabncp} and
  leave the other cases to the reader.

  Consider $(G,G') = (\Sp(2n,\bR),\rO(p,q))$, $W = M_{p,n}\times
  M_{q,n}$, $\fpp'^* = M_{p,q}$ and $p, q \geq 2n$. For $(A,B)
    \in M_{p,n}\times M_{q,n} = W$, $\phi'(A,B) = A B^T$ has rank $n$
    if and only if $A$ and $B$ have rank $n$. This proves (i) and the
    equality in (ii).

    Let $x'\in \phi'(W^\circ)$.  Let $(A,B), (A',B') \in
    \phi'^{-1}(x') \cap W^\circ$.  We have
\begin{equation}\label{eq:ab1}
AB^T = \phi'(A,B) = x' =\phi'(A',B') = A'(B')^T.
\end{equation} 
Here $x'$, $A$, $B$, $A'$, $B'$ are all rank $n$ matrices.  Since the
column space of $A$ (resp. $A'$) is same as the column space of $x'$,
we may assume that $A = A'$ by the action of $K_\bC = \GL(n,\bC)$. If
we interpret $A : \bC^n \rightarrow \bC^p$ as an injective linear map,
it is clear that~\eqref{eq:ab1} implies $B^T = B'^T$. This proves that
$(A,B)$ and $(A',B')$ are in the same $K_\bC$-orbit.  Hence
$\phi'^{-1}(x')\cap W^\circ = \phi'^{-1}(x')$ is a single
$K_\bC$-orbit.

Next suppose $k \in K_\bC$ stabilizes $(A,B)$. Hence $k \cdot A = A
k^{-1} = A$.  Since $A$ is an injective map, $k$ is the identity element.  This shows
that the $K_\bC$-action is faithful. This proves~(ii).

Let $x\in \phi(W^\circ)$.  Let $(A,B), (A',B') \in \phi^{-1}(x)\cap
W^\circ$. We have
\[
(A^TA,B^TB) = \phi(A,B) = x =\phi(A',B') = (A'^TA',B'^TB').
\]
Since $\Ker A = \Ker A' = 0$ and $A^TA = A'^TA'$, there is an $o\in \rO(p,\bC)$ such that
$A = o A'$ by Witt's theorem (for example, see
\cite{Howe95}*{Theorem~3.7.1}). The same argument applies to $B$ and
$B'$. Hence $\phi^{-1}(x)\cap W^\circ$ is a single orbit of $K'_\bC =
\rO(p,\bC)\times \rO(q,\bC)$. This proves~(iii).

Part (iv) follows from (i), (ii) and (iii).
\end{proof}

\begin{thm} \label{thm:orbits} Let $(G,G')$ be a reductive dual pair
  in the stable range where $G$ is the smaller member as in
  Table~\ref{tab:sstabncp}.  Let $\cO \in \frakN_{K_\bC}(\fpp^*)$. Set
  $\cO' = \phi'(\phi^{-1}(\cO)\cap W^\circ)$. Then
\begin{enumerate}[(i)]
\item $\phi'^{-1}(\cO') = \phi'^{-1}(\cO')
  \cap W^\circ = \phi^{-1}(\cO)\cap W^\circ = \phi'^{-1}(\cO')\cap
  \phi^{-1}(\bcO)$  is a $K_\bC \times K'_\bC$-orbit.
\item $\cO'$ is a $K'_\bC$-orbit.
\item $\phi(\phi'^{-1}(\cO'))= \cO$.
\item $\cO' = \theta(\cO)$.
\end{enumerate} 
\end{thm}

Using Table~4 in \cite{Daszkiewicz:2005}, one may calculate the above
orbits in $\phi^{-1}(\cO)$ and $\phi'(\phi^{-1}(\cO))$ explicitly and
verify the theorem directly. However we will sketch a simpler proof
below.

\begin{proof}[Sketch of the proof of Theorem \ref{thm:orbits}]
  Parts (i) to (iii) are direct consequences of
  Lemma~\ref{lem:fullrank}.

  By Theorem~2.5 in \cite{NOZ1}, $\phi^{-1}(\bcO)$ has a unique open
  dense $K_\bC\times K'_\bC$-orbit $\cD$.  Since $\phi^{-1}(\cO)\cap
  W^\circ$ is open and nonempty in $\phi^{-1}(\bcO)$, it is equal to
  $\cD$ and $\phi^{-1}(\bcO) = \overline{\cD}$. Hence $\bcOp \supseteq
  \phi'(\overline{\cD}) = \phi'(\phi^{-1}(\bcO)) \supseteq \cO'$ so
  $\bcOp = \theta(\bcO)$. This proves (iv).
\end{proof}

\subsection{} \label{SA4} {\sc Proof of Proposition \ref{prop:alpha}.}
Suppose $x = \phi(w)$ and $x' = \phi'(w)$ as in Proposition
~\ref{prop:alpha}. Then $w \in W^\circ$ by Theorem
\ref{thm:orbits}(i).  Fix a $k'\in K'_{x'}$. Then $(k')^{-1} \cdot w
\in \phi'^{-1}(x')$. By Lemma~\ref{lem:fullrank}(ii) there is a unique
$k \in K_\bC$ such that $k\cdot w = (k')^{-1} \cdot w$.  Since $x =
\phi((k')^{-1} \cdot w) = \phi(k\cdot w) = k\cdot \phi(w) = k\cdot x$,
we have $k \in K_x$. We define $\alpha(k') = k$.  It is
straightforward to check that $\alpha$ is a group homomorphism
and~\eqref{eq:stab} holds.

  Now we prove that $\alpha$ is surjective. Fix a $k \in K_{x}$. Since
  $k\cdot w \in \phi^{-1}(x)\cap W^\circ$, there is an element $k'\in
  K'_\bC$, such that $k'\cdot w = k \cdot w$ by Lemma
  \ref{lem:fullrank}(iii). It is clear that that $k' \in K'_{x'}$ so
  $k\cdot w = \alpha(k')^{-1} \cdot w$. Since the $K_\bC$-action on
  $\phi'^{-1}(x')$ is free, we have $\alpha(k')^{-1} = k$. This proves
  that $\alpha$ is surjective. \qed

\subsection{} \label{sec:A2} We discuss the scheme theoretical
properties of the moment maps. 

Let $R = W-W^\circ$ be the set of matrices without full rank.  Let
$\cN = \phi^{-1}(0)\cap W^\circ $ and $\partial \cN := \bcN - \cN$. By
Theorem~\ref{thm:orbits}, $\cN$ is a single $K'_\bC$-orbit, $\bcN =
\phi^{-1}(0)$ and $\partial \cN = R\cap \bcN$. We state some well
known geometric properties of the null fiber $\bcN$.

\begin{thm}[\cite{DT} \cite{Kostant} \cite{NOZ1}]
\label{thm:Kos} Let $(G,G')$ be a
  real reductive dual pair in the stable range as in Table~\ref{tab:sstabncp}.
\begin{enumerate}[(i)]
\item We have $\bC[W] = \cH \otimes \cS(\fpp)$ where $\cH$ is the
  space of $K'_\bC$-harmonic. In particular, the map $\phi\colon \Wcp
  \to \fpp^*$ is a faithfully flat morphism.  All the fibers of
    $\phi$ have the same dimension (see for example, the discussion on
    page 239 in \cite{PV}).

\item The scheme theoretical fiber $W\times_{\fpp^*}\Set{0}$ is
  reduced, i.e. $\bcN = W\times_{\fpp^*}\Set{0}$. 

\item If the dual pair is not \eqref{eq:ddagger} in Section
  \ref{S13}, then $\bcN$ is normal and $\partial \cN$ has
    codimension at least $2$ in~$\bcN$. \qed
\end{enumerate}
\end{thm}

We state Proposition~11.3.13(ii) in \cite{EGA:IV.3} which we will need
later in the proof of Lemma~\ref{lem:rednorm}.

\begin{prop} \label{PA6} Suppose $f\colon X_1 \to X_2 $ is a
    finitely presented flat morphism of schemes. Let $x_1 \in X$ and $x_2
    = f(x_1) \in X_2$. Then $x_1$ is reduced (resp. normal) in $X_1$ if
\begin{enumerate}[(i)]
\item  $x_2$ is reduced (resp. normal) in $X_2$ and 
\item $x_1$ is reduced (resp. normal) in $X_1 \times_{X_2} \Set{x_2}$. \qed
\end{enumerate} 
\end{prop}

Let $\cO$ be a nilpotent $K_\bC$-orbit in $\fpp^*$. Let $Z:=
W\times_{\fpp^*}\overline{\cO}$ (resp. $Y := W\times_{\fpp^*}\cO$)
be the scheme theoretic inverse image  of $\bcO$ (resp. $\cO$).

\begin{lemma} \label{lem:rednorm}
\begin{enumerate}[(i)]
\item The schemes $Z$ and $Y$ are reduced. 
\item Suppose the dual pair is not \eqref{eq:ddagger}. Then $Y$ is
  normal. If $\bcO$ is normal, then $Z$ is normal. 
\end{enumerate}
\end{lemma}

By the above lemma, we can also view $Z = \phi^{-1}(\overline{\cO})$
and $Y = \phi^{-1}(\cO)$ as the set theoretical inverse images.

\begin{proof}
  Our base field is $\bC$ so geometrically reduced
  (resp. geometrically normal) is equivalent to reduced
  (resp. normal).

  (i) Let $E_r$ (resp. $E_n$) be the set of elements in $W$ which is
  geometrically reduced (resp. geometrically normal) in the fiber of
  $\phi(w)$, i.e.
\[
E_r:= \Set{w\in W| w \text{ is geometrically reduced in } W\times_{\fpp^*}
  \phi(w)}.
\]
Since $\phi\colon W \to \fpp^*$ is faithfully flat, $E_r$
(resp. $E_n$) is open in $W$ by Theorem 12.1.1(vii)
(resp. Theorem~12.1.6(iv)) in~\cite{EGA:IV.3}. By
Theorem~\ref{thm:Kos}(ii) and (iii), $\bcN \subseteq E_r$ (resp. $\bcN
\subseteq E_n$).

\smallskip

We claim that $Z \subseteq E_r$ (resp. $Z \subseteq E_n$). We only prove
$Z \subseteq E_r$. The proof of $Z\subseteq E_n$ is the same.  Since
$E_r$ is open and $Z$ is closed, it suffices to prove that $E_r$
contains every closed point $z \in Z$. Indeed let $z \in
Z$ be a closed point. Since $\phi\colon W\to \fpp^*$ is
an affine quotient map, it maps a $K_\bC$-invariant closed subset in
$W$ to a closed subset in $\fpp^*$ (see Corollary 4.6 in \cite{PV}).
Therefore \[ \phi(\overline{K_\bC K'_\bC \cdot z}) =
\overline{\phi(K_\bC K'_\bC \cdot z)} = \overline{K_\bC \phi(z)} = \bcO \ni
0.
\]
Hence $\emptyset \neq \overline{K_\bC K'_\bC \cdot z} \cap \bcN
\subset \overline{K_\bC K'_\bC \cdot z}\cap E_r$. The subset $E_r$ is
open so $(K_\bC K'_\bC \cdot z) \cap E_r \neq \emptyset$. Finally $E_r$
contains $z$ because it is $K_\bC\times K'_\bC$-invariant.  This
proves our claim.

\smallskip

We note that $\phi|_Z \colon W\times_{\fpp^*} \bcO \to \bcO$ and
$\phi|_Y \colon W\times_{\fpp^*} \cO \to \cO$ are faithfully flat.
Since $\bcO$ is reduced and $Y \subseteq Z \subseteq E_r$, applying Proposition
  \ref{PA6} to $X_1 = Z$ (resp. $X_1 = Y$) and $X_2 = \bcO$ proves that $Z$ (resp. $Y$) is
  reduced. This gives (i).

  Since $\cO$ is smooth, it is normal. By assumption $\bcO$ is normal.
  The proof of (ii) follows a similar argument as that of (i). This completes
the proof of Lemma~\ref{lem:rednorm}.
\end{proof}

We state a consequence of Lemma \ref{lem:rednorm}.

\begin{lemma} \label{L21}
We have  $\bcOp = Z/K_\bC$, or equivalently, $\bC[\bcOp] =
  (\bC[W] \otimes_{S(\fpp)}\bC[\bcO])^{K_\bC}$.
\end{lemma}
\begin{proof}
  By Lemma \ref{lem:rednorm}(i), $\bC[Z] = (\bC[W]
  \otimes_{S(\fpp)}\bC[\bcO])$ is reduced.  By \eqref{eq22} and
  \cite{W}, $\phi'$ is an affine quotient map onto its image so
  $\bC[Z]^{K_\bC} = \bC[\bcOp]$. Also see Proposition~3(2) in
  \cite{Ohta:1991C}. This proves lemma.
\end{proof}

\smallskip

By Theorem~2.5 in \cite{NOZ1}, $Y = \phi^{-1}(\cO)$ contains an open dense $K_\bC\times K'_\bC$-orbit $Z^\circ$.  
Let $\partial Z^\circ = Y-Z^\circ$. By Theorem~\ref{thm:orbits}(i), $\partial Z^\circ = R\cap Y$ where $R =
W-W^\circ$  is the set of elements without full rank. 

\begin{lemma} \label{lem:codim2} Suppose the dual pair $(G,G')$ is in
  the stable range where $G$ is the smaller member and we exclude the
  dual pairs \eqref{eq:ddagger}. Then $\codim(Y, \partial Z^\circ)
  \geq 2$.
\end{lemma}

\begin{proof}
  If $\partial Z^\circ = \emptyset$, then there is nothing to prove.
  Now suppose $\partial Z^\circ \neq \emptyset$. Let $C =K_\bC\times
  K'_\bC$ and let $C_0 = K_0\times K'_0$ be its connected component.
  Since $C$ may not be connected, $Z^\circ$, $Y$ and $Z$ may not be
  irreducible.  We decompose $Z^\circ = \bigsqcup_{j\in J} Z_j^\circ$ in to
  $C_0$-orbits.  Each $Z_j^\circ$ is irreducible. Since $C/C_0$
  permutes $\set{Z_j^\circ|j\in J}$, $\dim Z^\circ = \dim Z_j^\circ$
  for all $j\in J$.  Let $Z_j = \overline{Z_j^\circ}$ in~$W$ and let
  $Y_j = Z_j\cap Y$. Then $Z_j$ and $Y_j$ are irreducible and
  $C_0$-invariant. In fact $Z = \bigcup_{j\in J} Z_j$ and $Y =
  \bigcup_{j\in J} Y_j$ are the decompositions of $Z$ and $Y$
  respectively into irreducible components.

  Let $\bcN = \phi^{-1}(0)$ be the closed null cone and let $\partial
  \cN = \bcN - \cN$. In the stable range, it is known that $\partial
  \cN = R \cap \bcN$ . Furthermore, if the dual pair is not
  \eqref{eq:ddagger}, then $\codim(\bcN,\partial \cN) = \dim \bcN
  -\dim \partial \cN \geq 2$.

Consider $\phi|_Z\colon Z\to \bcO$. By Theorem~\ref{thm:Kos}~(i), we have
  \[
  \dim Y = \dim Z = \dim \bcO +\dim \bcN. 
  \]
 
We claim that $\dim \partial Z^\circ \leq \dim Y -2$ which will prove
the lemma.  It suffices to show that for any closed point
$z\in \partial Z^\circ$, $\dim_z \partial Z^\circ \leq \dim Y -2$.
Here $\dim_v V$ denotes the Krull dimension of the local ring $\cO_{V,v}$ at a point $v$ in a variety $V$.

 We
  consider the morphism $\phi|_{R\cap Z} \colon R\cap Z \to
  \bcO$.
  By the semi-continuity
  of fiber dimension, the set $E = \set{w\in R\cap
    Z|\dim_w\phi|_{R\cap Z}^{-1}(w) \leq
    \dim \partial \cN}$ is open.
  For $z\in \partial Z^\circ$, let $S= K_\bC K'_\bC z$ be the orbit of $z$. Since $\overline{S}$ is
  $K'_\bC$-invariant and closed, the discussion after \eqref{eq22}
  shows that $\phi(\overline{S})$ is closed. In fact
  $\phi(\overline{S}) = \bcO$, because $\phi(\overline{S})$ is closed and contains
  $\phi(\partial Z^\circ) = \cO$.
  This implies $0 \in
  \phi(\overline{S})$ and $\emptyset \neq \overline{S} \cap
  \bcN\subseteq R\cap \bcN\subseteq
  \pcN \subseteq E$.
  Therefore $z\in E$, i.e. 
  $\dim_z \phi|_{R\cap Z}^{-1}(\phi(z))\leq \dim \pcN$. Hence
  \[
\begin{split}
  \dim_z \partial Z^\circ \leq& \dim_z R\cap Z \leq \dim_{\phi(z)}\bcO
  + \dim_z \phi|_{R\cap
    Z}^{-1}(\phi(z)) \\
  \leq& \dim \bcO + \dim \partial \cN \leq \dim \bcO +\dim \bcN -2 =
  \dim Y -2.
\end{split}
\]
This proves the claim and the lemma. 
\end{proof}

\section{Invariants of contragredient representations}
\label{sec:dual}

In this section, we state some known facts about the invariants of
contragredient representations. Since the proofs are not easily
available elsewhere, we supply them as well.

Let $\sfG$ be a real reductive group with complexified Lie algebra $\sfg$ and let $\sfK$ be a maximal compact subgroup
of $\sfG$. Let $(\varrho,V)$ be a $(\sfg,\sfK)$-module of finite
length. Let $(\varrho^*, V^*)$ be its contragredient representation
where $V^* = \Hom(V,\bC)_{\sfK-\text{finite}}$.

We recall the variety $\VC(V)$ associated to the annihilator
ideal $\Ann V = \Ann_{\cU(\sfg)}V$.  It is a subvariety in the
nilpotent cone of $\sfg^*$ cut out by the graded ideal $\Gr(\Ann V)$.

\begin{prop} \label{P24}
We have $\VC(V^*) = \VC(V)$.
\end{prop} 
\begin{proof}
  Let $\iota$ be the anti-involution on $\cU(\sfg)$ such that
  $\iota(X) = -X$ and $\iota(XY) = YX$ for all $X, Y \in \sfg$.
  Passing to the graded module $\bC[\sfg^*] = S(\sfg) =\Gr\cU(\sfg)$,
  $\Gr\iota$ is given by pre-composing the map on $\sfg^*$ defined by
  $\sfg^*\ni\lambda\mapsto -\lambda$. Then $\iota(\Ann V) = \Ann V^*$
  and $\VC(V^*) = -\VC(V)$. On the other hand, $\VC(V)$ is a union of
  nilpotent $\sfG_\bC$-orbits so $\VC(V)= -\VC( V)$. This proves the
  proposition.
\end{proof}

\subsection{} \label{SB1}  Let~$G$ and $K$
as in Table \ref{tab:sstabncp}. Let $\wtG$ and $\wtK$ be their
respective inverse images in $\wtSp(W_\bR)$. We relate the associated cycles of an
irreducible Harish-Chandra module of $\wtG$ and its contragredient module. By Proposition~4.I.8 in
\cite{MVW}, Theorem 2.4 in \cite{Sun} and \cite{LST}, there is an
automorphism $C \in \Aut(\wtG)$ such that for all semisimple $g \in
\tG$, $C(g)$ is conjugate to $g^{-1}$ in $\wtG$.  By replacing $C$
with $\Ad(\tilde{g}) \circ C$ for some $\tilde{g} \in \wtG$ if
necessary, we may further assume that $C$ stabilizes $\wtK$ and a
Cartan subgroup of $\wtK$. Hence $\Ad_C(\fkk) = \fkk$ and $\Ad_C(\fpp)
= \fpp$.  We call $C$ a {\it dualizing automorphism}.  If $\wtG$ is
the trivial double cover of a connected real algebraic group, then we
may choose~$C$ to be the Chevalley involution~\cite{Adams}.

Let $(\varrho, V)$ be an irreducible $(\fgg, \wtK)$-module.  We define
a representation $(\varrho^C, V^C)$ where $V^C = V$, $\varrho^C(k) =
\varrho(C(k))$ for all $k \in \wtK$ and $\varrho^C(X) =
\varrho(\Ad_C(X))$ for all $X \in \cU(\fgg)$. Then $(\varrho^C, V^C)$
is isomorphic to the contragredient representation $(\varrho^*, V^*)$
(c.f.  Corollary~1.2 in~\cite{Adams} and Theorem 3.1 in \cite{Sun}).

If $\cO$ is a nilpotent $\wtK_\bC$-orbit in $\fpp^*$ generated by $x$,
then $\Ad_C^*(\cO)$ is a nilpotent $\wtK_\bC$-orbit in $\fpp^*$
generated by $y := \Ad_C^*(x)$. We recall that $\wtK_x$ is the stabilizer of $x$ in $\wtK_\bC$. Then $\wtK_y =
C(\wtK_x)$. If $\chi_x$ is a $\wtK_x$-module
(resp. $\wtK_x$-character), then $\chi_x\circ C$ is a $\wtK_y$-module
(resp. $\wtK_y$-character).

\begin{prop} 
We have
\begin{enumerate}[(i)]
\item $\AV(\varrho^*) = \Ad_C^*(\AV(\varrho))$ and 

\item $\AC(\varrho^*) = \Ad_C^*(\AC(\varrho))$. 

\item Suppose $x\in \fpp^*$ generates an open orbit in $\AV(\varrho)$.
  Let $\chi_x$ be the isotropy character of $\varrho$ at $x$.  Then
  $\chi_{x}\circ C$ is the isotropy character of $\varrho^*$ at
  $\Ad_C^*(x)$.
\end{enumerate}
\end{prop}

\begin{proof}
  Let $\{ V_j \}_{j \in \bN}$ be an good filtration of $(\varrho,V)$.
  Then $\{ V_j \}_{j \in \bN}$ is also a good filtration of
  $(\varrho^C, V^C)$ since $C(\wtK) = \wtK$ and $\Ad_C(\fgg) = \fgg$.
  Therefore the $(\cS(\fpp), \wtK_\bC)$ action on $\Gr V^C$ is given
  by pre-composing $C$, i.e.  $\Gr V^C = \Gr V\circ C$. This proves
  the lemma.
\end{proof}

\subsection{} \label{SB2} Let $(\rho, V)$ be an irreducible
$(\fgg,\wtK)$-module which is a quotient of $\sY$. Let \linebreak $\Set{V_j=
\cU_j(\fgg)V_\tau'}_{j \in \bN}$  be the filtration generated by lowest degree
$\wtK$-type $V_\tau$. For a regular semisimple element $k$ in a Cartan
subgroup of $\wtK$, one can show that $C(k)$ is $\wtK$-conjugate to
$k^{-1}$. This implies that $\tau \circ C|_{\wtK} \simeq \tau^*$. We
fix a $(\fgg,\wtK)$-module isomorphism between $V^C$ and $V^*$. Since
$V_\tau$ has multiplicity~$1$ in $V$, $V_{\tau^*}$ has
multiplicity~$1$ in $V^*$ too.  We set $V_j^C :=
V_j$ and $V^*_j := \cU_j(\fgg)V_{\tau^*}$. Therefore the filtration $\Set{ V_j^C}_{j \in \bN}$ defined on $V^C$ is the same as the filtration $\Set{ V^*_j}_{j \in \bN} $ defined on $V^*$.

\end{document}